\documentclass[12pt]{amsart}
\usepackage[utf8]{inputenc}

\usepackage{amsmath,amssymb,amsbsy,amsfonts,amsthm,latexsym,mathabx,amsopn,amstext,amsxtra,euscript,amscd,stmaryrd,mathrsfs,cite,array,mathtools,enumerate
}
\usepackage{enumitem}

\usepackage{hyperref}
\hypersetup{
    colorlinks=true,
    linkcolor=blue,
    filecolor=magenta,      
    urlcolor=cyan,
    citecolor=magenta,
    pdftitle={Smooth polynomials},
    pdfpagemode=FullScreen,
    }

\usepackage{a4wide}

\newtheorem{theorem}{Theorem}
\newtheorem{lemma}[theorem]{Lemma}

\newtheorem{cor}[theorem]{Corollary}

\newcommand{\F}{\mathbb{F}}
\newcommand{\K}{\mathbb{K}}
\newcommand{\ZZ}{\mathbb{Z}}

\newcommand{\Fq}{\mathbb{F}_q}
\newcommand{\e}{\mathbf{e}}
\newcommand{\mand}{ \quad \text{and} \quad}

\renewcommand{\d}{\mathrm{d}}
 \renewcommand{\d}{\mathrm{d}}
 
\newcommand{\re}[1]{\mathrm{Re}#1}

 \usepackage{todonotes}

\colorlet{LightGrey}{white!70!black} 
 
\def\cA{{\mathcal A}}
\def\cB{{\mathcal B}}

\def\cI{{\mathcal I}}
\def\cJ{{\mathcal J}}

\def\cM{{\mathcal M}}

\def\cP{{\mathcal P}}

\def\cR{{\mathcal R}}
\def\cS{{\mathcal S}}

\def\cU{{\mathcal U}}
\def\cV{{\mathcal V}}

\def\bT{{\mathbf T}}

\def\fM{{\mathfrak M}}
\def\fm{{\mathfrak m}}

\title{Smooth polynomials with several prescribed coefficients}
 \author[L.~M{\'e}rai]{L{\'a}szl{\'o} M{\'e}rai}
 \address{Department of Computer Algebra, Eötvös Loránd University,  Pázmány P. sétány 1/c, H-1117, Budapest, Hungary} 
 \email{merai@inf.elte.hu}

\allowdisplaybreaks


\date{\today}
\begin{document}

\maketitle
\begin{abstract}
Let $\Fq[t]$ be the polynomial ring over the finite field $\Fq$ of $q$ elements. A polynomial in $\Fq[t]$ is called $m$-smooth (or $m$-friable) if all its irreducible factors are of degree at most $m$. In this paper, we investigate the distribution of $m$-smooth (or $m$-friable) polynomials with prescribed coefficients. Our technique is based on character sum estimates on smooth (friable) polynomials, Bourgains's argument (2015) applied for polynomials by Ha (2016) and on double character sums on smooth (friable) polynomials.
\end{abstract}

\section{Introduction}

In recent years, many spectacular results have been obtained on important problems combining some arithmetic properties of the integers and some conditions on their digits in a given basis, see for example \cite{cite:Bourgain,DES,Dr14,DrMaRi19,green,MaRi10,May19,Swa20,TuPa}. 

In particular, Bourgain \cite{cite:Bourgain} investigated the distribution of primes with prescribed bits. Later, Swaenepoel \cite{Swa20} extended Bourgain's result for arbitrary base. Namely, for a given positive integer $k$, consider its $g$-ary expansion
$$
k=\sum_{j\geq 0}k_ig^i, \quad k_i\in\{0,1,\dots, g-1\}.
$$
Then, there is a $\delta(g)\in(0,1)$, that for a given index-set $\cI\subset \{0,1,\dots, n-1\}$ with $\#\cI \leq \delta(g) \cdot  n $ and $0\in \cI$, 
the number of prime numbers with prescribed digits   $\alpha_i\in \{0,1,\dots, g-1\}$ ($i\in \cI$) with $\alpha_0\neq 0$ satisfies
$$
\#\left\{ p<g^n, p_i=\alpha_i, i\in \cI , \text{$p$ prime}   \right\}\sim \frac{\pi(g^n)}{g^{\#\cI}},
$$
where $\pi(x)$ is the prime counting function.

Ha \cite{Ha:Irreducible_polynomials} investigated the same problem over function fields. Namely, let $\Fq$ be the finite field of $q$ elements, and consider the polynomial ring $\Fq[t]$ over $\Fq$. Let $\cP(n)$ be the set of \emph{monic irreducible} polynomials of degree $n$. It is well-known, that the number $\pi_q(n)=\#\cP(n)$ of such irreducible polynomials  satisfies
\begin{equation}\label{eq:irred}
 \frac{q^n}{n}-2\frac{q^{n/2}}{n}   \leq  \pi_q(n)\leq  \frac{q^n}{n},
\end{equation}
see e.g. \cite[Lemma~4]{pollack13}. Improving earlier results on coefficients of irreducible polynomials \cite{cite:Ham,pollack13,cite:Wan}, Ha showed that the number of irreducible polynomials with preassigned coefficients is proportional to $\pi_q(n)$. In particular, let $\cM(n)$ be the set of \emph{monic} polynomials of degree $n$ and for an index-set $\cI\subset \{0,1\dots, n-1\}$ and $\alpha_i\in\Fq$ ($i\in \cI$) let $\cJ$ be the set of monic polynomials of degree $n$ with $i$th coefficient $\alpha_i$
\begin{equation}\label{eq:setJ}
\cJ=\left\{   t^n +\sum_{i=0}^{n-1} c_it^i\in\cM(n): c_i=\alpha_i , i\in \cI \right\}.
\end{equation}

Write 
$$
\delta =\frac{\#\cI}{n},
$$
then for $\delta < 1/4$ we have
\begin{equation*}
    \#(\cP(n)\cap \cJ  ) =\frac{\mathfrak{G} q^{n-\#\cI}}{n}\left(1 + O\left(\frac{1- \log_q \delta}{q^{1/\delta - 4/(\delta +1)}} \right) \right) + O(q^{3n/4}),
\end{equation*}
where
$$
\mathfrak{G}=\left\{
\begin{array}{cl}
   1  & \text{if $0\not\in \cI$}, \\
   1+\frac{1}{q-1}  & \text{if $0\in \cI$ and $\alpha_0\neq 0$},\\
   0  & \text{if $0\in \cI$ and $\alpha_0= 0$}
\end{array}
\right.
$$
and the implied constants are absolute.
In particular, we have the expected frequency  of irreducible polynomials with prescribed coefficients, $ \#(\cP(n)\cap \cJ  ) \sim \frac{\mathfrak{G} q^{n-\#\cI}}{n}$, if $q\rightarrow \infty$ or $\delta = \#\cI/n\rightarrow 0$.
\smallskip

In this paper, we investigate the distribution of \emph{smooth} (or \emph{friable}) polynomials with preassigned coefficients. For a given integer $m$, a polynomial is called $m$-smooth (or $m$-friable) if all its irreducible divisors are of degree at most $m$. Write
$$
\cS(n,m)=\{f\in \cM(n): f \text{ is $m$-smooth}\}\mand \Psi(n,m)=\#\cS(n,m).
$$

If we let
$$
u=\frac{n}{m},
$$
it is known, that the proportion of $m$-smooth polynomials of degree $n$ tends to Dickman's $\rho$ function defined via the delayed differential equation 
$$
\tau\rho'(\tau)+   \rho(\tau -1)=0
$$ 
for $\tau \geq 1$ with initial condition $\rho(\tau)=1$ for $\tau\in[0,1]$.
It is decreasing, and de Bruijn proved \cite{deBruijn} that it satisfies
$$
\rho(u)=\exp\left( -u \log\left( u \log u\right) +u+ O\left(\frac{u \log \log u}{\log u}\right) \right)
$$
uniformly for $u\geq 3$.

Then Manstavi\u{c}ius  \cite{Man} proved that
\begin{equation}\label{eq:Psi}
\Psi(n,m)=q^n \rho\left(u\right)\left(1+O\left(\frac{u\log(u+1) }{m}\right)\right)
\end{equation}
holds in range $n\geq m \geq \sqrt{n\log n}$
(see also \cite{Car,GarPa,PaGoFl98}). Recently, Gorodetsky  \cite{Gor22+} gave the following better approximation of $\Psi(n,m)$ 
\begin{equation}\label{eq:Psi2}
    \Psi(n,m)=q^n \rho\left(\frac{n}{m}-\frac{n}{2m^2}\right)\exp\left(O_\varepsilon\left(  \frac{u\log^2(u+1)}{m^2}+\frac{1}{u} \right)\right) 
\end{equation}
uniformly for $(2+\varepsilon)\log_q n \leq m\leq n$ and
\begin{equation}\label{eq:Psi3}
    \Psi(n,m)=q^n \rho\left(\frac{n}{m}-\frac{n}{2m^2}\right)\left(1+O\left(\frac{\log(u+1) }{m}\right)\right)
\end{equation}
if 
\begin{equation}\label{eq:n_range}
\sqrt{n\log n}=O(m)
\end{equation}
see \cite[Corollary~1.6.]{Gor22+}. Moreover, the error terms in \eqref{eq:Psi} and \eqref{eq:Psi3} are optimal.

The main result of this paper investigates the proportion of $m$-smooth polynomials with prescribed coefficients.

\begin{theorem}\label{thm:1}
Let $0<\varepsilon<1/4$. 
Let $\cI\subset\{0,1,\dots, n-1\}$ and  $\alpha_i\in\Fq$  ($i\in \cI$) such that
$0\in \cI$ and $\alpha_0\neq 0$.
Assume,
\begin{equation}\label{eq:condition1}
 m\leq n/100, \quad \text{that is,} \quad    u\geq 100.
\end{equation}
Write $\delta =\#\cI/n$ and assume that
\begin{equation*}
\delta < 1/24, \quad \text{that is,} \quad  \#\cI < n/24.
\end{equation*}
Then for
\begin{equation}\label{eq:thm1:range}
 (2+\varepsilon)\log_q n \leq m\leq n   ,
\end{equation}
we have
\begin{equation*}
\begin{split}
\#(\cS(n,m)\cap \cJ) &= 
\frac{\Psi(n,m)}{q^{\# \cI}} 
\left(1 + O\left( \exp\left(O_\varepsilon \left(\frac{u\log^2(u+1)}{m^2}\right)\right) \frac{\delta^{-1} \log \delta^{-1}}{q^{C /\delta}}\right)\right) \\
&\qquad
 +
O\left(
 \frac{q^{(1-\eta)n}}{q^{\#\cI}}
\right)
,
\end{split}
\end{equation*}
holds for some positive $C>0$ and $\eta>0$,
where the implied constant may depend on $\varepsilon$.
\end{theorem}
We remark, that the constant $1/24$ is not optimized and arises from technical aspects of the proof; with additional analysis this bound could be improved.

As the first error term in
Theorem~\ref{thm:1} might be greater than 1, it gives bounds on $\#(\cS(n,m)\cap \cJ)$ in the range \eqref{eq:thm1:range}. However, in a shorter range for $m$ instead of \eqref{eq:thm1:range}, that is, if $\sqrt{n\log n}=O(m)$, we also have an asymptotic result.

\begin{cor}\label{cor:thm1}
Having the assumptions as in Theorem~\ref{thm:1}, for
\begin{equation}\label{eq:Om}
 \sqrt{n\log n}=O(m) ,  
\end{equation}
we have
\begin{equation}\label{eq:thm1}
\begin{split}
\#(\cS(n,m)\cap \cJ) &= 
\frac{\Psi(n,m)}{q^{\# \cI}} 
\left(1 + 
\frac{\delta^{-1} \log \delta}{q^{C /\delta}}\left( 
 1+O\left(
 \frac{\log(u+1)}{m}
 \right)
 \right)\right) \\
&\qquad
 +
O\left(
 \frac{q^{(1-\eta)n}}{q^{\#\cI}}
\right)
,
\end{split}
\end{equation}
holds for some positive $C>0$ and $\eta>0$.
\end{cor}

Indeed, under the condition \eqref{eq:Om}, one has
$\exp\left(O_\varepsilon \left(\frac{u\log^2(u+1)}{m^2} \right)\right) =1+O\left(
 \frac{\log(u+1)}{m}
 \right)$.

In particular, the number of $m$-smooth  polynomials with prescribed coefficients has the expected frequency if $q\rightarrow \infty$ or $\delta \rightarrow 0$.

\begin{cor}\label{cor:1}
Assume that the conditions of Corollary~\ref{cor:thm1} hold. Moreover, if 
\begin{itemize}
    \item for fixed $n$, $u$ and $\delta$, we have $q\rightarrow \infty$; or 
    \item for fixed $q$, we have $n\rightarrow \infty$ and $\delta 
    =o(1)$, 
\end{itemize}
then
\begin{equation}\label{eq:expected}
\#(\cS(n,m)\cap \cJ) \sim  q^{-\# \cI}\Psi(n,m)  .
\end{equation}
\end{cor}

If $\alpha_0=0$, then $\#(\cS(n,m)\cap \cJ)$ deviates from  the expected frequency. For example, if $I=\{0,1,\dots, r-1\}$ with $\alpha_i=0$ for $i\in \cI$, then
$$
\cS(n,m)\cap \cJ = \{t^nf: f\in\cS(n-r,m)  \} 
$$
and thus by \eqref{eq:Psi},
\begin{equation}\label{eq:zero_coeff}
\#(\cS(n,m)\cap \cJ)= \Psi(n-r,m)\sim \rho\left(\frac{n-r}{m}\right)q^{n-r}
\end{equation}
in contrast to \eqref{eq:thm1}.

In order to handle the general case write
\begin{equation*}
I=\{i_0<\dots< i_\nu<\dots\}\quad  \text{with  } \quad \alpha_{i_0}=\dots =\alpha_{i_{\nu-1}}=0\neq \alpha_{i_{\nu}}
\end{equation*}
and for $0<\kappa<n$ define
\begin{align*}
\Lambda(n,m,\cI,\kappa)= &\sum_{\substack{0\leq j<\nu \\ i_j<\kappa}} 
\frac{1}{q^{n-j-1 }}  
\left(
\Psi(n-i_j-1,m)-\frac{1}{q}\Psi(n-i_j,m)
\right)\\
&-\mathbf{1}(\nu<\kappa)
\frac{1}{q^{n-\nu-1}}  \left(
\frac{1}{q-1} \Psi(n-i_\nu-1,m)-\frac{1}{q(q-1)}\Psi(n-i_\nu,m)
\right).
\end{align*}
Then we have

\begin{theorem}\label{thm:2} 
Let $0<\varepsilon<1/4$. 
Let $\cI\subset\{0,1,\dots, n-1\}$ and  $\alpha_i\in\Fq$  ($i\in \cI$).
\begin{equation}\label{eq:prop:condition1}
 m\leq n/100, \quad \text{that is,} \quad    u\geq 100.
\end{equation}
Write $\delta =\#\cI/n$ and assume that
\begin{equation*}
\delta < 1/24, \quad \text{that is,} \quad  \#\cI < n/24.
\end{equation*}
Then for
\begin{equation}\label{eq:prop:range}
 (2+\varepsilon)\log_q n \leq m\leq n   ,
\end{equation}
and 
\begin{equation}\label{eq:kappa}
    0<\kappa \leq n/5,
\end{equation}
we have
\begin{equation*}
\begin{split}
\#(\cS(n,m)\cap \cJ) &= 
\frac{\Psi(n,m)}{q^{\# \cI}} 
\left(1 + O\left( \exp\left(O_\varepsilon \left(\frac{u\log^2(u+1)}{m^2}\right)\right) \frac{\delta^{-1} \log \delta^{-1}}{q^{C /\delta}}\right)\right) \\
&\quad
 + \frac{q^n}{q^{\#\cI}}\left(\Lambda(n,m,\cI,\kappa)+O\left(  q^{\left(-\frac{1}{2}+\varepsilon\right)n+3\kappa}\right)\right) +
O\left(
 m^2 n^{1/2}q^{n-\kappa/2}
\right)
,
\end{split}
\end{equation*}
holds for some positive $C>0$,
where the implied constant may depend on $\varepsilon$.
\end{theorem}
Again, if $m$ is in the shorter range \eqref{eq:Om} instead of \eqref{eq:prop:range}, an asymptotic result can be obtained for $\#(\cS(n,m)\cap \cJ)$.
\begin{cor}
Having the assumptions as in Theorem~\ref{thm:2}, in the range \eqref{eq:Om} we have
\begin{equation*}
\begin{split}
\#(\cS(n,m)\cap \cJ) &= 
\frac{\Psi(n,m)}{q^{\# \cI}} 
\left(1 + 
\frac{\delta^{-1} \log \delta}{q^{C /\delta}}\left( 
 1+O\left(
 \frac{\log(u+1)}{m}
 \right)
 \right)\right) \\
&\quad
 + \frac{q^n}{q^{\#\cI}}\left(\Lambda(n,m,\cI,\kappa)+O\left( q^{\left(-\frac{1}{2}+\varepsilon\right)n+3\kappa}\right)\right) +
O\left(
 m^2 n^{1/2}q^{n-\kappa/2}
\right)
,
\end{split}
\end{equation*}
holds for some positive $C>0$,
where the implied constant may depend on $\varepsilon$.
\end{cor}

By Theorem~\ref{thm:2}, the distribution of $m$-smooth polynomials only depends on the position of the prescribed zero coefficients. 

In particular, we remark that if $0\in \cI$, then we can apply either Theorem~\ref{thm:1} if $\alpha_0\neq 0$ or \eqref{eq:zero_coeff} if $\alpha_0=0$.
In the case $0\not \in \cI$, Theorem~\ref{thm:2} yields the expected frequency if $q\rightarrow \infty$. 
\begin{cor}\label{cor:2}
Assume that the conditions of Theorem~\ref{thm:2} hold and $n,m$ satisfy \eqref{eq:Om}. If 
$0\not \in \cI $ and $q\rightarrow \infty$ for fixed $n$ and $u$,
then \eqref{eq:expected} holds. 
\end{cor}

\smallskip

\subsection*{Related works}

Theorems \ref{thm:1} and \ref{thm:2} give asymptotics to the number smooth polynomials with several prescribed coefficients.

Cohen \cite{Cohen70,Cohen72} investigated the distribution of given factorization type for large field characteristic, i.e. when $q=p^r$ with large prime $p$. 

The distribution of polynomials with given factorization type as $q$ growing to infinity has been investigated, see e.g. Cohen \cite{Cohen70,Cohen72}, and Bank, Bary-Soroker, and Rosenzweig \cite{BSR}. These result give immediately the asymptotic for smooth polynomials where the right-most or left-most coefficients are prescribed as $q$ grows.

Recently, Kuttner and Wang \cite{KuttnerWang}, using different techniques, have considered the \emph{exact} number of smooth polynomials of degree $n$ with given trace, that is, when $\cI=\{n-1\}$.

Gorodetsky \cite{cite:Gorodetsky:d_anal}  (see also \cite{Man}) obtained an asymptotic formula for $\Psi(n,m)$ beyond the range \eqref{eq:n_range}, with a different main term that does not involve the Dickman's  $\rho$ function. The proofs of Theorems~\ref{thm:1} and~\ref{thm:2} rely on a detailed analysis of $\rho$. Incorporating Gorodetsky’s result into our framework would therefore require new ideas, and we leave this direction as an open problem.

Hauck and Shparlinski \cite{HaukSparlinski} have also considered similar problems over $\ZZ$ by proving the existence of smooth \emph{integers} (that is, integers without large prime factor) with few nonzero binary digits. 
In addition, Cumberbatch \cite{cumb} investigated  smooth integers with missing digits.

\bigskip

The proof of Theorem~\ref{thm:2} uses the circle method. The major arc analysis is based on character sum estimates on smooth polynomials,
see Section~\ref{sec:char_sum}, and 
Bourgain's argument~\cite{cite:Bourgain}, applied for polynomials over finite fields by Ha \cite{Ha:Irreducible_polynomials}, see Section~\ref{sec:SJ}.

The minor arc analysis is based on double exponential sums and the good
factorization properties of smooth polynomials, see Section~\ref{sec:minor}.

Finally, Theorem~\ref{thm:1} follows from Theorem~\ref{thm:2}, see Section~\ref{sec:proof}.

\section{Outline of the proof}

For $f\in \Fq[t]$, write
\begin{equation}\label{eq:norm}
|f|=q^{\deg f}
\end{equation}
with the convention $|0|=0$. One can extend the absolute value $|\cdot|$ to $\F_q(t)$ in the natural way by $|f/g|=q^{\deg f-\deg g}$.

Let $\K_{\infty}$ be the set of formal Laurent series of $1/t$
$$
\K_{\infty}=\left\{\xi=\sum_{i\leq k}x_it^i, k\in \ZZ, x_i\in\Fq \right\},
$$ 
which is the completion of $\Fq(t)$ at the prime associated to the $(1/t)$-adic valuation.

We extend the norm \eqref{eq:norm} to $\K_{\infty}$ by
$$
|\xi|=q^k,
$$
where $k$ is the largest index so that $x_k\neq 0$.

Also, for $\xi\in \K_{\infty}$ we define
$$
\| \xi\|=\min_{f\in \Fq[t] }|\xi-f|.
$$

Write
$$
\e(\xi)=\exp\left(\frac{2\pi i }{p} \mathrm{tr}_{\Fq}(x_{-1}) \right),
$$
where $\mathrm{tr}_{\Fq}$ is the (absolute) trace of $\Fq$ and $p$ is the characteristic of $\Fq$.

We define $\mathbf{T}$ by $\{\xi\in \K_{\infty}: |\xi|<1 \}$, and fix an additive Haar measure, normalized so that $\int_{\mathbf{T}} \mathrm{d}\xi=1$. Then we have for a polynomial $f\in\Fq[t]$,
\begin{equation}\label{eq:ortogonal}
 \int_{\mathbf{T}} \e(f\xi) \mathrm{d}\xi=
 \left\{
 \begin{array}{cl}
    1  & \text{if } f=0, \\
    0  &  \text{if } f\neq 0,
 \end{array}
 \right.
\end{equation}
see \cite[Theorem~3.5]{Hayes:Goldbach}.

For given integers $m\leq n$, write
$$
S(\xi;n,m)=\sum_{f\in \cS(n,m)}\e(f \xi).
$$

For a nonempty subset $\cI\subset \{0,1,\dots n-1\}$ fix $\alpha_i\in\F_q$ for $i\in \cI$. 
Let $\cJ$ defined by \eqref{eq:setJ} and write 
$$
S_{\cJ}(\xi)=\sum_{f\in\cJ}\e(f\xi).
$$

By \eqref{eq:ortogonal}, we have
\begin{align}\label{eq:int}
   \int_{\mathbf{T}} S(\xi;n,m) \overline{S_{\cJ}(\xi)}  \mathrm{d}\xi = \sum_{f\in\cS(n,m)} \sum_{g\in \cJ} \int_{\mathbf{T}}   \e((f-g)\xi) \mathrm{d}\xi =   \#(\cS(n,m)\cap \cJ).
\end{align}

We investigate this integral by the circle method. To do so, let $\ell$ be a positive parameter to be fixed later such that
\begin{equation}\label{eq:ell-def} 
\ell <\min\{n-\kappa,n/2 \}
\end{equation}
and define the major and minor arcs by
\begin{equation}\label{eq:major}
\fM = \bigcup_{|a|<|g|\leq q^{\kappa}}\left\{\xi \in \bT: \left|\xi-\frac{a}{g}\right|<q^{-n+\ell} \right\}
\mand \fm = \bT \setminus \fM.
\end{equation}

We give an approximation of $S(\xi;n,m)$ on major arcs in Section~\ref{sec:major} using character sum estimates on smooth polynomials introduced in Section~\ref{sec:char_sum}. 
The minor arc estimates of $S(\xi;n,m)$ are based on double exponential sums and the good
factorization properties of smooth polynomials, see Section~\ref{sec:minor}.

The analysis of $S_{\cJ}(\xi)$ is based on Bourgain's argument \cite{cite:Bourgain}, applied for polynomials over finite fields by Ha \cite{Ha:Irreducible_polynomials}, see Section~\ref{sec:SJ}.

\section{Preliminaries}

\subsection{Notations}
For given functions $F$ and $G$, the notations $F\ll G$, $G \gg F$ and $F =O(G)$ are all equivalent to the statement that the inequality $|F| \leq C|G|$
holds with some constant $C > 0$. Throughout the paper, any implied
constants in symbols O, $\ll$ and $\gg$ are absolute, unless specified otherwise.
For a set $\cA$ we denote the characteristic function of $\cA$ by $\mathbf{1}(\cA)$.  
\smallskip

We let $\cM=\bigcup_n \cM(n)$ and $\cP=\bigcup_n \cP(n)$ to be the sets of \emph{all monic} and \emph{all monic irreducible polynomials} respectively.

\subsection{Arithmetically distributed relations}
For a monic polynomial $f=t^n+c_{n-1}t^{n-1}+\dots +c_0$, we call $c_{n-1},\dots, c_{n-\ell}$ the first $\ell$ coefficients of $f$. Define $f^{*}=t^{\deg f}f(1/t)$. For $\ell \geq 0$ and $g\in \Fq[t]$, define the \emph{arithmetically distributed relation}
$\cR_{\ell,g}$ as
$$
f\equiv h \mod \cR_{\ell,g} \quad \text{if} \quad f\equiv h \mod g \mand f^*\equiv h^* \mod t^{\ell+1},
$$
that is, if $f$ and $h$ are congruent modulo $g$ and their first $\ell$ coefficients are the same.
Clearly, an arithmetically distributed relation $\cR_{\ell,g}$ is an equivalence relation.

We also write $\cR_{\ell}=\cR_{\ell,1}$ and $\cR_{g}=\cR_{0,g}$.

Observe that $\cR_\ell$ on fixed degree polynomials captures the polynomials on intervals, that is, for polynomials $f,h\in\cM(n)$,
$$
f \equiv h \mod \cR_\ell \quad \text{if and only if} \quad |f-h|<q^{n-\ell}.
$$

The polynomial $f$ is invertible modulo $\cR_{\ell,g}$ if $\gcd(f,g)=1$. These invertible polynomials form a group, denoted by $G_{\ell, g}=\left(\cM/\cR_{\ell,g}\right)^\times$. Write $G_{\ell}=\left(\cM/\cR_{\ell}\right)^\times$ and $G_{g}=\left(\cM/\cR_{g}\right)^\times$. Then we have 
\begin{equation}\label{eq:CRT}
G_{\ell, g} \cong G_{\ell} \times G_{g},
\end{equation}
see \cite[Theorem~8.6]{Hayes:Distribution} and \cite[Lemma~1.1]{Hsu:Distribution}. We also have $\# G_{\ell}=q^\ell$ and
$$
\#G_{g}=\#(\Fq[t]/(g))^{\times}=\Phi(g).
$$
We have the following lower bound on $\Phi(g)$, see \cite[Lemma~2.3]{Ha:Irreducible_polynomials},
\begin{equation}\label{eq:phi}
\Phi(g) \gg 
\frac{|g|}{\log_q (1+\deg g) +1 } \quad \text{for } g\neq 0.
\end{equation}

Let $\chi$ be a character modulo an arithmetically distributed relation $\cR_{\ell,g}$, and let $L(s,\chi)$ be the associated $L$-function,
$$
L(s,\chi)=\sum_{f\in\cM}\frac{\chi(f)}{|f|^s}=\prod_{\substack{\omega \nmid g\\\omega \in\cP}}\left(1-\frac{\chi(\omega)}{|\omega|^s}\right)^{-1} , \quad \re (s)  >1.
$$

The following result, see \cite[Theorem~1.3]{Hsu:Distribution}, follows from the  Riemann hypothesis for the $L$-functions of algebraic curves over finite fields proved by Weil \cite{weil}.

\begin{lemma}\label{lemma:Weil}
If $\chi$ is a non-trivial multiplicative character modulo $\cR_{\ell,g}$, then there exists a polynomial $P_\chi$ of degree at most $ \ell -1+\deg g$ and of form
$$
P_\chi(z)=\prod_i\left(1-\alpha_i z\right)
$$
with $L(s,\chi)=P_\chi(q^{-s})$ and $|\alpha_i|=\sqrt{q}$ or 
$|\alpha_i|=1$. 
\end{lemma}

Using Lemma~\ref{lemma:Weil}, on can derive the following character sum estimates along irreducible polynomials, see \cite{Hsu:Distribution} for more details.

\begin{lemma}\label{lemma:char_sum_irred}
    Let $\chi$ be a non-trivial character modulo $\cR_{\ell,g}$. Then 
    $$
\left|\sum_{\omega\in\cP(n)}\chi(\omega)\right|\leq \frac{1}{n}(\ell-1+\deg g)q^{n/2}.
    $$
\end{lemma}

\subsection{Exponential sums}

We have, see \cite[Theorem~3.7]{Hayes:Goldbach}.
\begin{lemma}\label{lemma:ort1}
    Suppose $|\xi|<1$. Then
    $$
\sum_{f\in \cM(n)}\e(f\xi)=
\left\{
\begin{array}{lc}
q^n \e(t^{n}\xi) & \text{if } |\xi|<q^{n},\\
0 & \text{otherwise}.
\end{array}
\right.
    $$
\end{lemma}

We have, see \cite[Equation~(4.6)]{Hayes:Goldbach}.
\begin{lemma}\label{lemma:ort2}
If $a,b$ are invertible modulo $\cR_{\ell, g}$, then 
$$
\frac{1}{q^{\ell}\Phi(g)}
\sum_{\chi \bmod \cR_{\ell, g}}\bar{\chi}(a)\chi(b)=
\left\{
\begin{array}{ll}
1  & \text{if } a\equiv b \mod \cR_{\ell,g},\\
0 & \text{otherwise}.
\end{array}
\right.
    $$
\end{lemma}

For the following result see \cite[Theorem~6.1]{Hayes:Goldbach}.
\begin{lemma}\label{lemma:Ram}
For polynomials $f,g$, we have
$$
\sideset{}{^*}\sum_{a \bmod g}\e(af/g)=\sum_{d \mid \gcd(f,g)}|d|\mu(g/d).
$$
\end{lemma}

We can prove the following results.

\begin{lemma}\label{lemma:Gauss}
For polynomials $b,g$, we have
$$
\sum_{\chi \bmod \cR_{\ell, g}}
\left|\ 
\sideset{}{^*}\sum_{\substack{a \bmod \cR_{\ell,g}\\ a \equiv b \bmod \cR_\ell }}\chi(a)\e(a/g)
\right|^2=q^{\ell} \Phi(g)^2.
$$
\end{lemma}
\begin{proof}
We have by Lemma \ref{lemma:ort2}, that
\begin{align*}
&\sum_{\chi \bmod \cR_{\ell, g}}
\left|\ 
\sideset{}{^*}\sum_{\substack{a \bmod \cR_{\ell,g}\\ a \equiv b \bmod \cR_\ell }}\chi(a)\e(a/g)
\right|^2\\
&=
\sum_{\chi \bmod \cR_{\ell, g}}
\left(\  \sideset{}{^*}\sum_{\substack{a_1 \bmod \cR_{\ell,g}\\ a_1 \equiv b \bmod \cR_\ell }}\chi(a_1)\e(a_1/g) \right)\left( \ \sideset{}{^*}\sum_{\substack{a_2 \bmod \cR_{\ell,g}\\ a_2 \equiv b \bmod \cR_\ell }}\bar{\chi}(a_2)\e(-a_2/g) \right)\\
&=\sideset{}{^*}\sum_{\substack{a_1,a_2 \bmod \cR_{\ell, g}\\ a_1\equiv a_2\equiv b \mod \cR_\ell }} \e((a_1-a_2)/g)\sum_{\chi \bmod \cR_{\ell, g}} \chi(a_1)\bar{\chi}(a_2)\\
&=q^\ell \Phi(g) \sideset{}{^*}\sum_{
\substack{a_1\equiv a_2 \bmod \cR_{g}\\a_1\equiv a_2\equiv b \bmod \cR_{\ell}}} 
\e((a_1-a_2)/g)= q^{\ell}\Phi(g)^2.
\end{align*}
\end{proof}

For the following estimate see \cite[Lemma~4]{cite:Merai_Divisors_of_sum}.

\begin{lemma}\label{lemma:minor-1}
Let 
\begin{equation*}
\xi= \frac{a}{g}+ \gamma \quad \text{with} \quad |\gamma| \leq \frac{1}{|g|^2} \quad \text{and} \quad (a,g)=1.
\end{equation*}
Then we have
\begin{equation*}
    \sum_{f\in\cM(k)}\left|\sum_{h \in \cM(j)} \e(\xi fh) \right|\ll \frac{q^{k+j}}{|g|}  + q^k +|g|.
\end{equation*}

\end{lemma}

\subsection{Dickman's $\rho$ function}

We need the following result on Dickman's $\rho$ function which follows directly from \cite[Lemma~III. 5.11]{cite:Tenenbaum} and \cite[Corollary~III. 5.15]{cite:Tenenbaum}.

\begin{lemma}\label{lemma:rho}
 Uniformly over $0\leq v\leq u$ and $u\geq 3$, we have
 $$
 \rho(u-v)\ll (u \log u)^{v}\rho(u).
 $$
\end{lemma}

\subsection{Further preliminaries}

The following lemma of Hayes \cite[Theorem 4.3]{Hayes:Goldbach} is an analogue of a well-known result of Dirichlet
in the theory of Diophantine approximation.

\begin{lemma}\label{lemma:dio}
    For any integer $n$ and $\xi\in\mathbf{T}$, there is a pair of coprime polynomials $a,g\in\Fq[t]$ with $g$ monic, $|a|<|g|\leq q^{n/2}$ such that
    $$
\left|\xi - \frac{a}{g}\right|<\frac{1}{|g|q^{n/2}}.
    $$
\end{lemma}

Let $\tau(f)$ denote the number of divisors of $f$. We have the following bound on $\tau(f)$.

\begin{lemma}\label{lemma:number_of_divisors}
Fo any $\varepsilon>0$ we have for any polynomial $f$ of degree $\deg f\geq 2$ that
$$
\tau(f) \ll  |f|^{(2+\varepsilon)/\log \deg f},
$$
where the implied constant may depend on $\varepsilon$.
\end{lemma}
The proof of this lemma is analogue to the integer case, however, we include its proof for the sake for completeness.
\begin{proof}
Write
$$
f=g_1^{e_1}\dots g_k^{e_k}
$$
with pairwise distinct irreducible polynomials $g_1,\dots, g_k$.
Then
$$
\tau(f)=\prod_{i=1}^k(e_i+1).
$$
For any $\lambda>0$ we have
$$
\frac{\tau(f)}{\exp(\lambda \deg f)}
=\prod_{i=1}^k\frac{e_i+1}{\exp(\lambda e_i\deg g_i)}
\leq 
\prod_{i=1}^k\frac{e_i+1}{1+\lambda e_i\deg g_i }.
$$
For $\deg g_i >1/\lambda$ we have
$$
\frac{e_i+1}{1+\lambda e_i\deg g_i }<1.
$$
Thus
\begin{equation}\label{eq:tau_empty}
\frac{\tau(f)}{\exp(\lambda \deg f)}\leq 
\prod_{\deg g_i \leq 1/\lambda}\frac{e_i+1}{1+\lambda e_i\deg g_i }.
\end{equation}
Put
$$
\lambda = \frac{2\log q}{\log \deg f}.
$$
If $\lambda\geq 1$, then the right-hand side of \eqref{eq:tau_empty} is an empty product, thus
$$
\tau(f)\leq \exp\left( \frac{2\log q \deg f}{\log \deg f}\right)\leq |f|^{2/\log \deg f}
$$
and the result follows.

If $0<\lambda<1$, then 
\begin{equation}\label{eq:lambda}
\frac{e_i+1}{1+\lambda e_i \deg e_i}\leq 
\frac{1}{\lambda }.
\end{equation}

As there are at most $q^{1/\lambda}$ irreducible polynomials of degree at most $1/\lambda$, we get by \eqref{eq:tau_empty} and \eqref{eq:lambda}, that
\begin{equation}\label{eq:tau}
\frac{\tau(f)}{\exp(\lambda \deg f)}\leq \left(\frac{1}{\lambda}\right)^{q^{1/\lambda}}=\exp \left(q^{1/\lambda} \log(1/\lambda) \right).
\end{equation}
Substituting \eqref{eq:lambda} to it, we get that its logarithm is
$$
\log \frac{\tau(f)}{\exp\left(\frac{2\log q\deg f}{\log \deg f} \right)}\leq (\deg f)^{1/2}\log \frac{\log \deg f}{2\log q}\leq \varepsilon  \frac{(\log q)( \deg f)}{\log \deg f}
$$
if either $\deg f$ or $q$ is large enough in terms of $\varepsilon$.
\end{proof}

\section{Distribution of smooth polynomials}
\label{sec:char_sum}
The purpose of this section is to prove character sum estimates along smooth polynomials. The proof is based on the method of Bhowmick, L\^{e} and Liu \cite{cite:Bhow_et_al}, see also \cite{cite:Gorod2022}. We consider both cases that $q$ or $n$ goes to infinity, in contrast to \cite{cite:Bhow_et_al,cite:Gorod2022}.

\begin{lemma}\label{lemma:char_sum}
Let $\chi\neq\chi_0$ be a character modulo $\cR_{\ell,g}$ and $n>\ell-1+\deg g$. 
\begin{enumerate}[label=(\roman*)]
    \item \label{item:charsum1} We have
$$
\left| \sum_{f\in S(n,m)}\chi(f)\right|\leq 
q^{\frac{n}{2}+\frac{\log \log(\ell+1+\deg g)}{\log(\ell+1+\deg g)}n}e^{ 39q \frac{\ell+1+\deg g}{\log^2(\ell+1+\deg g)} }
$$
if $\ell + \deg g$ is large enough in terms of $q$. One can choose 
$$
\ell + \deg g > 10^4 \mand  \frac{\log (\ell + 1+\deg g)}{\log \log (\ell + 1+\deg g)}>\log q .
$$
\item \label{item:charsum2} Let $0<\varepsilon<1/4$, 
then we have
$$
\left|\sum_{f\in S(n,m)}\chi(f)\right|\leq 
q^{\left(\frac{1}{2}+\varepsilon\right)n}e^{ \varepsilon (\ell+1+\deg g)}
$$
if $q$ is large enough in terms of $\varepsilon$. One can choose $q>(1+1/\varepsilon)^{1/\varepsilon}$.
\end{enumerate}
\end{lemma}

\begin{proof}
Write
$$
T(z,\chi;m)=\prod_{\omega \in \cP , \deg \omega \leq m}\left(1-\chi(\omega)z^{\deg \omega}\right)^{-1}.
$$
For $z\in\mathbb{C}$ with $|z|<1$, we have
$$
T(z,\chi;m)=\sum_{r=0}^\infty S(\chi;r,m)z^r, \quad S(\chi;r,m)=\sum_{f\in\cS(r,m)}\chi(f).
$$

The logarithm of $T(z,\chi;m)$ is given by
\begin{equation}\label{eq:logT}
\log T(z,\chi;m) = \sum_{\omega \in \cP ,\deg \omega \leq m } \chi(\omega) z^{\deg \omega}+ \sum_{k\geq 2 }k^{-1}\sum_{\omega \in \cP ,\deg \omega \leq m } \chi(\omega)^k z^{k\deg \omega}
.
\end{equation}
By Lemma~\ref{lemma:char_sum_irred}, we have
\begin{equation}\label{eq:charsum-nontrivial}
\left|\sum_{\omega\in\cP(r)}\chi(\omega) \right|\leq  \frac{\ell -1+ \deg g}{r}q^{r/2}.
\end{equation}
We also have the trivial bound
\begin{equation}\label{eq:charsum-trivial}
\left|\sum_{\omega\in\cP(r)}\chi(\omega) \right|\leq  \frac{q^{r}}{r}
\end{equation}
by \eqref{eq:irred}.

Let $R$ be a positive integer to be fixed later.

For $r\geq R$ we use \eqref{eq:charsum-nontrivial} in \eqref{eq:logT}, while for $r< R$ we use the trivial bound \eqref{eq:charsum-trivial}. Namely, \eqref{eq:charsum-nontrivial} yields that for $r\geq R$, we have that the coefficient of $z^r$ in \eqref{eq:logT} is bounded in absolute value by
\begin{equation}\label{eq:large_r}
    \frac{\ell -1+ \deg g}{r}q^{r/2} + \sum_{\substack{k\mid r\\k\geq 2}} k^{-1} \#\cP(r/k)\leq \frac{\ell + 1+\deg g}{r}q^{r/2}.
\end{equation}
For $r< R$, the trivial bound \eqref{eq:charsum-trivial} yields, that the coefficient of $z^r$ in \eqref{eq:logT} is bounded in absolute value by
\begin{equation}\label{eq:small_r}
    \frac{q^{r}}{r} + \sum_{\substack{k\mid r\\k\geq 2}} k^{-1} \#\cP(r/k)\leq \frac{2}{r}q^{r}.
\end{equation}
Combining \eqref{eq:large_r} and \eqref{eq:small_r}, we get
\begin{equation}\label{eq:log_T}
   \left| \log T(z,\chi;m) \right| \leq \sum_{1\leq r < R} \frac{2}{r}(q|z|)^r + \sum_{r\geq R} \frac{\ell+1+\deg g}{r}(q^{1/2}|z|)^r.    
\end{equation}
In order to prove \ref{item:charsum1}, put
$$
R=\left\lfloor 2\frac{\log(\ell + 1 +\deg g)}{\log q} \right\rfloor +1.
$$
With this choice, we have
\begin{equation}\label{eq:q-to-R}
q^{(R-1)/2}\leq \ell + 1+\deg g< q^{
R/2}.
\end{equation}
Also, let 
\begin{equation}\label{eq:delta_def}
0 <\delta<\min\left\{\frac{1}{\log q},\frac{1}{4} \right\}   
\end{equation}
to be chosen later. With this choice, put $Z=q^{-1/2-\delta}$ and let $|z|=Z$. 

Then for the first term of \eqref{eq:log_T} we have
 \begin{align}\label{eq:r_sum1}
     \sum_{1\leq r < R} \frac{2}{r}(q|z|)^r
     \leq 2 (qZ)^{R-1} \sum_{r\geq 0} (qZ)^{-r}
     \leq  \frac{2 (\ell+1+\deg g)^2  Z^{R-1}}{1-(qZ)^{-1}}
      \leq  13 (\ell+1+\deg g)^2 Z^{R-1}
 \end{align}
 using $|qZ|\geq q^{1/4}\geq 2^{1/4}$.
 
Also,    
\begin{align}\label{eq:r_sum2}
\sum_{r\geq R} \frac{\ell+1+\deg g}{r}(q^{1/2}Z)^r\leq  \frac{\ell+1+\deg g}{R} \cdot\frac{(q^{1/2}Z)^{R}}{1-q^{1/2}Z}. 
\end{align}
Then combining \eqref{eq:log_T}, \eqref{eq:r_sum1} and \eqref{eq:r_sum2}, we
have by \eqref{eq:q-to-R} that 
\begin{align*}
 \left| \log T(z,\chi;m) \right| &\leq 13 (\ell+1+\deg g)^2 q^{-(1/2+\delta)(R-1)} + \frac{\ell+1+\deg g}{R} \cdot\frac{q^{-\delta R}}{1-q^{-\delta}}\\
&\leq (\ell+1+\deg g)^{(1-2\delta)} \left( 13q^{1/2+\delta} +  \frac{1}{R(1-q^{-\delta})}\right).
\end{align*}
As $\delta<1/\log q$, we have $q^\delta<3$ and $1-q^{-\delta}\geq \frac{1}{2}\delta \log q$, thus
\begin{align*}
 \left| \log T(z,\chi;m) \right|&\leq (\ell+1+\deg g)^{(1-2\delta)} \left( 39q^{1/2} +  \frac{2}{R\delta \log q}\right)\\
&\leq (\ell+1+\deg g)^{(1-2\delta)} \left( 39q^{1/2} +  \frac{2}{\delta \log ( \ell+1+\deg g )}\right)
 .    
\end{align*}
Let $C_Z$ denote the circle centered at 0 with radius $Z$ traversed anti-clockwise. Then, by the Cauchy's integral formula, we have  
\begin{align*}
    |S(\chi,n,m)|&=\left|\frac{1}{2\pi i} \int_{C_Z} T(z,\chi;m)z^{-n-1}\d z \right|\\
    &\leq \max_{C_Z}|T(z,\chi;m)| Z^{-n}\\
    &\leq \exp\left((\ell+1+\deg g)^{(1-2\delta)} \left( 39q^{1/2} +  \frac{2}{\delta \log ( \ell+1+\deg g )}\right) \right) q^{(1/2+\delta)n}.
\end{align*}
Put 
$$
\delta = \frac{\log \log(\ell+1+\deg g )}{\log(\ell+1+\deg g ) }
$$
and observe that it satisfies \eqref{eq:delta_def} by assumption: 
If $\ell+1+\deg g >10^4$, then $\delta< 1/4$, and also, $\delta < \frac{1}{\log q}$. Then
\begin{align*}
&(\ell+1+\deg g)^{(1-2\delta)} \left( 39q^{1/2} +  \frac{2}{\delta \log ( \ell+1+\deg g )}\right)\\
=&\frac{\ell+1+\deg g }{\log ^2  (\ell+1+\deg g)}\left( 39q^{1/2} +  \frac{2}{\log \log ( \ell+1+\deg g )}\right)\\
  \leq   &
  39q\frac{\ell+1+\deg g }{\log ^2  (\ell+1+\deg g)}
\end{align*}
which proves \ref{item:charsum1}.

In order to prove \ref{item:charsum2}, we choose $R=1$, that is, in \eqref{eq:log_T} the first term vanishes and
\begin{equation*}
       \left| \log T(z,\chi;m) \right| \leq  \sum_{r\geq 1} \frac{\ell+1+\deg g}{r}(q^{1/2}|z|)^r.  
\end{equation*}
Choose $Z=q^{-1/2-\varepsilon}$. For  $|z|=Z$, we have
\begin{equation*}
       \left| \log T(z,\chi;m) \right| \leq  \sum_{r\geq 1} \frac{\ell+1+\deg g}{r}(q^{-\varepsilon})^r
       \leq   (\ell+1+\deg g) \cdot\frac{q^{-\varepsilon }
 }{1-q^{-\varepsilon}}
 \leq \varepsilon (\ell+1+\deg g)
\end{equation*}
if $q$ is large enough in terms of $\varepsilon$.

Then, by the Cauchy's integral formula, we have  
\begin{align*}
    |S(\chi,n,m)|&=\left|\frac{1}{2\pi i} \int_{C_Z} T(z,\chi;m)z^{-n-1}\d z \right|\\
    &\leq \max_{C_Z}|T(z,\chi;m)| Z^{-n}\\
    &\leq \exp\left(\varepsilon (\ell+1+\deg g) \right) q^{(1/2+\varepsilon)n},
\end{align*}
which proves \ref{item:charsum2}.
\end{proof}

The following result unifies the two bounds of Lemma~\ref{lemma:char_sum}.

\begin{cor}\label{cor:char_sum}
Let $\chi\neq\chi_0$ be a character modulo $\cR_{\ell,g}$ and $n\geq \ell+\deg g$. For all $0<\varepsilon<1/4$, we have
$$
\sum_{f\in S(n,m)}\chi(f)\ll 
q^{\left(\frac{1}{2}+\varepsilon\right)n}
,
$$
where the implied constant may depend on $\varepsilon$.
\end{cor}
\begin{proof}
We can assume that $q^n$ is large enough in terms of $\varepsilon$.

    If $q>(1+1/\varepsilon)^{1/\varepsilon}$, then we apply \ref{item:charsum2} of Lemma~\ref{lemma:char_sum} directly by enlarging $\varepsilon$. If $q<(1+1/\varepsilon)^{1/\varepsilon}$ and $\ell+\deg g$ is small in terms of $q$ (and hence in term of $\varepsilon$), then the bound follows. Otherwise, 
    $$
    e^{ 39q \frac{\ell+1+\deg g}{\log^2(\ell+1+\deg g)} } \ll e^{\varepsilon (\ell+\deg g)}
    $$
    and if $n$ is large enough, then 
    $$
q^{\frac{\log \log(\ell+1+\deg g)}{\log(\ell+1+\deg g)}} \ll q^{\varepsilon n}.
    $$
 Then the result follows by enlarging $\varepsilon$.   
\end{proof}

As an application, we prove the following result on the distribution of smooth polynomials in intervals.

\begin{lemma}\label{lemma:smooth_in_interval}
Let $0<\varepsilon<1/4$.
For $1\leq \ell\leq n$ we have for $b\in \Fq[t]$ that
$$
    \#\{f\in \cS(n,m): f\equiv b \mod \cR_\ell\}=\frac{1}{q^\ell}\Psi(n,m) + O\left(q^{\left(\frac{1}{2}+\varepsilon\right)n}
    \right),
$$
where the implied constant may depend on $\varepsilon$.
\end{lemma}
\begin{proof}
    We have
    \begin{align*}
       \#\{f\in \cS(n,m): f\equiv b \mod \cR_\ell\}&=
       \frac{1}{q^\ell}\sum_{\chi \bmod \cR_\ell}\overline{\chi(b)} \sum_{f\in \cS(n,m)}\chi(f).
    \end{align*}
  The contribution of the principal character $\chi_0$ is
  $$
  \frac{1}{q^\ell}\overline{\chi_0(b)} \sum_{f\in \cS(n,m)}\chi_0(f) = \frac{1}{q^\ell}\Psi(n,m) .
  $$
By Corollary~\ref{cor:char_sum}, the contribution of the non-principal characters is
\begin{align*}
   \frac{1}{q^\ell}\sum_{\chi \neq \chi_0}\overline{\chi(b)} \sum_{f\in \cS(n,m)}\chi(f)    \ll 
   q^{\left(\frac{1}{2}+\varepsilon\right)n}
   .
   \end{align*} 
\end{proof}

\section{Analysis of $\cS(\xi;n,m)$}

In this section we investigate $\cS(\xi;n,m)$ on major and minor arcs.

\subsection{Major arc estimates}\label{sec:major}

Let $\Psi_g(n,m)$ denote the number of $m$-smooth polynomials of degree $n$ which are coprime to $g$,
$$
\Psi_g(n,m)=\#\{f\in\cS(n,m): \gcd(f,g)=1\}.
$$

We split the summation over $f$ in $\cS(\xi;n,m)$ into residue classes modulo $\cR_\ell$
$$
\cS(\xi;n,m) = \sum_{b\bmod \cR_\ell}\sum_{\substack{f\in\cS(n,m)\\ f \equiv b \bmod \cR_\ell}}\e\left(f \xi\right)
$$
and investigate each term on the major arcs.

\begin{lemma}\label{lemma:major2}
Let $0<\varepsilon<1/4$. For $a,b\in \Fq[t]$, $\gcd(a,g)=1$, $\ell + \deg g<n$, we have
\begin{align*}
\sum_{\substack{f\in\cS(n,m)\\ f \equiv b \bmod \cR_\ell}}\e\left(\frac{af}{g}\right)&=
\sum_{d \mid g}   \frac{\mu(g/d)}{q^{\ell}\Phi(g/d)}  
\Psi_{g/d}(n-\deg d,m)
+O\left(q^{\left(\frac{1}{2}+\varepsilon\right)n}
|g|^{1/2}\right),
\end{align*}
where the implied constant may depend on $\varepsilon$.
\end{lemma}
\begin{proof}
We have
$$
\sum_{\substack{f\in\cS(n,m)\\ f \equiv b \bmod \cR_\ell}}\e\left(\frac{af}{g}\right)=\sum_{d\mid g} \sum_{\substack{f\in \cS(n-\deg d,m)\\ \gcd(f,g/d)=1\\f \equiv b \bmod \cR_\ell}}
\e\left(\frac{af}{g/d}\right).
$$
For a given $d$ we have by Lemma~\ref{lemma:ort2} that
\begin{align*}
\sum_{\substack{f\in \cS(n-\deg d,m)\\ \gcd(f,g/d)=1\\f \equiv b \bmod \cR_\ell}} \e\left(\frac{af}{g/d}\right)&
= \frac{1}{q^{\ell}\Phi(g/d)}\sum_{\chi \bmod \cR_{g/d,\ell}} \sum_{\substack{f\in \cS(n-\deg d,m)} } \chi(f)\sideset{}{^*} \sum_{\substack{c\bmod \cR_{g/d,\ell}\\ c \equiv b \bmod \cR_{\ell}}}\bar{\chi}(c) \e\left(\frac{ac}{g/d}\right),
\end{align*}
where in the innermost sum $c\in\cM(n-\deg d)$ runs over all reduced residue classes modulo $\cR_{g/d,\ell}$ such that $c\equiv b \mod \cR_{\ell}$. As $n-\deg d > \deg (g/d)+\ell$, 
by \eqref{eq:CRT}, the polynomial $c$ runs through all all residue classes modulo $g/d$.  Thus by Lemma~\ref{lemma:Ram}, the contributions of the principal characters $\chi_0^{(g/d,\ell)}$ modulo $\cR_{g/d,\ell}$ is
\begin{align*}
&\sum_{d \mid g}   \frac{1}{q^{\ell}\Phi(g/d)}  \sum_{\substack{f\in \cS(n-\deg d,m)}} \chi_0^{(g/d,\ell)}(f)
\sideset{}{^*} \sum_{c\bmod g/d}\e\left(\frac{ac}{g/d}\right)
=\sum_{d \mid g}   \frac{\mu(g/d)}{q^{\ell}\Phi(g/d)}  
\Psi_{g/d}(n-\deg d,m)
.
\end{align*}
By \eqref{eq:CRT} and Corollary~\ref{cor:char_sum}, the contribution of the non-principal characters are
\begin{align*}
&\frac{1}{q^{\ell}\Phi(g/d)}\left|\sum_{\substack{\chi^{(g/d, \ell)} \bmod \cR_{g/d,\ell}\\ \chi^{(g/d, \ell)}\neq \chi^{(g/d, \ell)}_0 }} \sum_{\substack{f\in \cS(n-\deg d,m) }} \chi^{(g/d, \ell)}(f)\sideset{}{^*} \sum_{\substack{c\bmod \cR_{g/d,\ell}\\ c\equiv b \bmod \cR_{\ell}}}\bar{\chi}^{(g/d, \ell)}(a) \e\left(\frac{ac}{g/d}\right)\right|\\
&\leq \frac{1}{q^{\ell}\Phi(g/d)}
\sum_{\substack{\chi^{(g/d, \ell)} \bmod \cR_{g/d,\ell}\\ \chi^{(g/d, \ell)}\neq \chi^{(g/d, \ell)}_0 }} \left|\sum_{\substack{f\in \cS(n-\deg d,m)}} \chi^{(g/d, \ell)}(f)\right|\left| \ \sideset{}{^*} \sum_{\substack{c\bmod \cR_{g/d,\ell}\\ c\equiv b \bmod \cR_{\ell}}}\bar{\chi}^{(g/d, \ell)}(a) \e\left(\frac{ac}{g/d}\right)\right|\\
&\ll 
q^{\left(\frac{1}{2}+\varepsilon\right)(n-\deg d)}e^{ \varepsilon (\ell + \deg g)} \frac{1}{q^\ell \Phi(g/d)}
\sum_{\substack{\chi^{(g/d, \ell)}\bmod \cR_{g/d,\ell}}} \left| \ \sideset{}{^*} \sum_{\substack{c\bmod \cR_{g/d,\ell}\\ c\equiv b \bmod \cR_{\ell}}}\bar{\chi}^{(g/d, \ell)}(a) \e\left(\frac{ac}{g/d}\right)\right|.
\end{align*}
By the Cauchy-Schwarz inequality and Lemma~\ref{lemma:Gauss} we have
\begin{align*}
&\left( \
\sum_{\substack{\chi^{(g/d, \ell)}\bmod \cR_{g/d,\ell}}} \left| \ \sideset{}{^*} \sum_{\substack{c\bmod \cR_{g/d,\ell}\\ c\equiv b \bmod \cR_{\ell}}}\bar{\chi}^{(g/d, \ell)}(a) \e\left(\frac{ac}{g/d}\right)\right|\right)^2\\
&\leq q^\ell \Phi(g/d) \sum_{\substack{\chi^{(g/d, \ell)}\bmod \cR_{g/d,\ell}}}  \left| \ \sideset{}{^*} \sum_{\substack{c\bmod \cR_{g/d,\ell}\\ c\equiv b \bmod \cR_{\ell}}}\bar{\chi}^{(g/d, \ell)}(a) \e\left(\frac{ac}{g/d}\right)\right|^2\\
&\leq q^{2\ell}\Phi(g/d)^3.
\end{align*}

Thus, summing over $d$ we get, that the error term is
$$
\ll 
q^{\left(\frac{1}{2}+\varepsilon\right)n}
\sum_{d\mid g}q^{-\left(\frac{1}{2}+\varepsilon\right)\deg d}\Phi(g/d)^{1/2}\ll q^{\left(\frac{1}{2}+\varepsilon\right)n}
|g|^{1/2}\sum_{d\mid g}\frac{1}{|d|^{1+\varepsilon}}
$$
As
\begin{equation*}
\sum_{d\mid g} \frac{1}{|d|^{1+\varepsilon}}= \sum_{i\leq \deg g} \frac{1}{q^{i(1+\varepsilon)}}\sum_{\substack{d\mid g\\ \deg d =i}} 1\leq \sum_{i\leq \deg g} \frac{1}{q^{\varepsilon i}}\leq \sum_{i=1}^{\infty} \frac{1}{q^{\varepsilon i}}\ll 1,
\end{equation*}
the result follows.
\end{proof}

The following result estimates the main term of Lemma~\ref{lemma:major2} for polynomials $g$ of form $g=t^k$.

\begin{lemma}\label{lemma:major}
Let $0<\varepsilon<1/4$. For $b\in \Fq[t]$, $\gcd(b,t)=1$, $1\leq k< n-\ell$, we have
\begin{align*}
\sum_{\substack{f\in\cS(n,m)\\ f \equiv b \bmod \cR_\ell}}\e\left(\frac{bf}{t^k}\right)=&\frac{1}{q^{\ell-1}(q-1)}\Psi(n-k,m) - \frac{1}{q^{\ell}(q-1)}\Psi(n-k+1,m)\\
&+O\left(q^{\left(\frac{1}{2}+\varepsilon\right)n+k/2}
\right)
,
\end{align*}
where the implied constant may depend on $\varepsilon$.
\end{lemma}

\begin{proof}
By~Lemma~\ref{lemma:major2}, the main term is

\begin{align*}
&\sum_{i=0}^{k} \frac{ \mu(t^{k-i}) }{q^\ell\Phi(t^{k-i})} \Psi_{t^{k-i}}(n-i,m)\\
&=\frac{1}{q^{\ell}}\Psi(n-k,m) - \frac{1}{q^{\ell}(q-1)}\Psi_t(n-k+1,m)\\
&=\frac{1}{q^{\ell}}\Psi(n-k,m) - \frac{1}{q^{\ell}(q-1)}\left(\Psi(n-k+1,m)-\Psi(n-k,m)\right)\\
&=\frac{1}{q^{\ell-1}(q-1)}\Psi(n-k,m) - \frac{1}{q^{\ell}(q-1)}\Psi(n-k+1,m).
\end{align*}
Then the result follows from Lemma~\ref{lemma:major2}.
\end{proof}

\subsection{Minor arc estimates}
\label{sec:minor}

\begin{lemma}\label{lemma:minor_counting}
For $\alpha, \beta \in \mathbf{T}$, we have
$$
\#\left\{f\in\cM(n): \|\alpha f  -\beta  \|< q^{-k} \right\}\leq \max\left\{1, q^{n-k} , q^{n}\|\alpha\|, q^{-k}/\|\alpha\| \right\}.
$$
\end{lemma}
\begin{proof}
Write
$$
\alpha =\sum_{i\leq 0}\alpha_i t^i \mand \beta =\sum_{i\leq 0}\beta_i t^i
$$
and let $\|\alpha\|=q^{-r}$.

If $\|\alpha f -\beta\|\leq q^{-k}$, then
\begin{equation}\label{eq:LES}
\sum_{ i=-j+r }^n f_i\alpha_{ -j-i} =\beta_{-j}, \quad \text{for } j=\max\{1,r-n+1\}, \dots, \min\{k,r \}
\end{equation}
As $\alpha_{-r}\neq 0$, the coefficient matrix has full rank, and thus \eqref{eq:LES} has 
$$
q^{n-(\min\{k,r\}-\max\{0, r-n\})}
$$ solutions.
\end{proof}

\begin{lemma}\label{lemma:minor_arcs}
Let $\xi=a/g+\gamma$ with $\gcd(a,g)=1$, $\|\gamma\|<1/|g|^2$. Then
 we have
$$
S(\xi;n,m)\ll 
\left\{
\begin{array}{c}
m^2 n^{1/2}\left(\frac{q^{n}}{|g|^{1/2}}+q^{3n/4+m/4}+|g|^{1/2}q^{(n+1)/2} \right),
 \\
 m^2 n^{1/2}\left(
|g|^{1/2}q^{3n/4+m/4+1/2}+\|g\gamma\|^{1/2}q^n  +\frac{q^{(n+1)/2}}{\|\gamma\|^{1/2}}     \right).
\end{array}
\right.
$$
\end{lemma}
\begin{proof}

Let us fix an arbitrary ordering $\preccurlyeq$ of polynomials which is compatible with the degree, that is, $g\preccurlyeq f$ whenever $\deg g < \deg f$. For a polynomial $f$ let $\Omega^+(f)$ and $\Omega^-(f)$ be the maximal resp. minimal irreducible divisor of $f$ with respect to the ordering $\preccurlyeq$. By collecting small degree divisors of $f$ into $u$, one can have a unique representation 
$f = u\cdot v$ of polynomials $f\in\cS(n,m)$ with polynomials $u,v$ satisfying
$$
\Omega^+(u) \preccurlyeq \Omega^-(v), \quad 
\frac{n-m}{2}\leq \deg(u)< \frac{n+m}{2},   \quad \text{and } 
\deg u + \deg \Omega^-(v) \geq \frac{n+m}{2}.
$$
For an irreducible polynomial $\omega \in \cP$ write
$$
\cU(j, \omega)=\left\{ u\in \cS(j,m): \, \Omega^+(u)= \omega \right\} 
$$
and
$$
\cV(j, \omega)=\left\{ v\in \cS(n-j,m) :\, \Omega^-(v)\geq  \omega, \deg \Omega^-(v) \geq \frac{n+m}{2}-j\right\}.
$$

Then, we have

\begin{align*}
S\left(\xi;n,m\right)
 &=
 \sum_{\frac{n-m}{2}\leq j <\frac{n+m}{2}} 
 \sum_{r=1}^{m} \sum_{\omega\in \cP(r)} \sum_{u\in\cU(j, \omega) } \sum_{v\in \cV(j, \omega) } \e\left(uv \xi\right).
 \end{align*}

Write
$$
S_{j,r} = \sum_{\omega\in \cP(r)} \sum_{u\in\cU(j, \omega) } \sum_{v\in \cV(j, \omega) }  \e\left(uv \xi\right).
$$

By the Cauchy inequality, we have
 \begin{align*}
|S_{j,r}|^2 
 &\ll \left(\sum_{\omega\in \cP(r)}\sum_{u\in \cU(j, \omega)}1\right)  \sum_{\omega\in \cP(r)} \sum_{u\in \cU(j, \omega)} \left|\sum_{v\in\cV(j, \omega)} \e\left(uv \xi\right)\right|^2 \\
 &\ll q^{j}  \max_{\omega\in\cP(r)}  \sum_{\substack{z \in \cM(j)}} \left|\sum_{v\in\cV(j, \omega)}\e \left( zv \xi\right)\right|^2 \\
 &  \ll q^{j}\max_{\omega\in\cP(r)}
 \sum_{v_1,v_2\in\cV(j, \omega)}\left| \sum_{\substack{z \in \cM(j)}} \e(z(v_1-v_2)\xi)\right|\\
 &  \ll q^{j}
 \sum_{v_1,v_2\in\cP(n-j)}\left| \sum_{\substack{z \in \cM(j)}} \e(z(v_1-v_2)\xi)\right|\\
   &  \ll q^{j}
  \sum_{b\in \Fq^*}
 \sum_{v\in\cP(n-j)}
 \sum_{i=0}^{n-j-1}
 \sum_{h \in \cM(i)}
 \left| \sum_{\substack{z \in \cM(j)}} \e\left(bz h \xi\right)\right|
 +q^{n+j}.
\end{align*}

To prove the first bound, we apply Lemma~\ref{lemma:minor-1} to get
\begin{align*}
  |S_{j,r}|^2 
 &\ll   q^j \sum_{b\in \Fq^*}
 \sum_{v\in\cP(n-j)}
 \sum_{i=0}^{n-j-1}
 \left(\frac{q^{i+j}}{|g|}+q^{i}+|g|\right) + q^{n+j}\\
 &\ll n \left(\frac{q^{2n}}{|g|}+q^{2n-j}+|g|q^{n+1} + q^{n+j}\right).
\end{align*}
Whence
$$
S(\xi;n,m)\ll m^2 n^{1/2}\left(\frac{q^{n}}{|g|^{1/2}}+q^{3n/4+m/4}+|g|^{1/2}q^{(n+1)/2} \right).
$$

To prove the second bound, we have by
Lemma~\ref{lemma:ort1}, 
\begin{equation}\label{eq:minor-1001}
\sum_{h \in \cM(i)} \left| \sum_{\substack{z \in \cM(j)}} \e\left(bz h \xi \right)\right|=
q^{j}\cdot \#\left\{h \in \cM(i): \| h \xi\|< q^{-j} \right\}.
\end{equation}

Write $h$ as $h=eg + f$, $|f|<|g|$. Then
$$
\left\|h \xi \right\| = \left\|e g\gamma + \frac{af}{g}  +f\gamma \right\| .
$$
Putting $\beta= af/g  +f\gamma $, then
by Lemma~\ref{lemma:minor_counting},
$$
\#\left\{e: \deg e =i-\deg g, \  \| e g \gamma +\beta\|< q^{-j} \right\} \ll \max\left\{1, \frac{q^{i-j}}{|g|}, \|g\gamma\|\frac{q^i}{|g|}, \frac{1}{\|g\gamma\|q^j}\right\}
$$
thus by \eqref{eq:minor-1001}
$$
\sum_{h \in \cM(i)} \left| \sum_{\substack{z \in \cM(j)}} \e\left(bz h \xi \right)\right|\ll |g|q^j + q^i+ \|g\gamma\|q^{i+j}  +\frac{1}{\|\gamma\|}.
$$
Whence
\begin{align*}
  |S_{j,r}|^2 
 &\ll   q^j \sum_{b\in \Fq^*}
 \sum_{v\in\cP(n-j)}
 \sum_{i=0}^{n-j-1}
\left(|g|q^j + q^i+ \|g\gamma\|q^{i+j}  + \frac{1}{\|\gamma\|}\right)+q^{n+j}\\
 &\ll n\left(|g|q^{n+j+1}+ q^{2n-j}+ \|g\gamma\|q^{2n}  +\frac{q^{n+1}}{\|\gamma\|}\right).
\end{align*}
It yields
$$
S(\xi;n,m)\ll m^2 n^{1/2}\left(
|g|^{1/2}q^{3n/4+m/4+1/2}+\|g\gamma\|^{1/2}q^n  +\frac{q^{(n+1)/2}}{\|\gamma\|^{1/2}}
\right).
$$
\end{proof}

\section{Analysis of $S_\cJ$}
\label{sec:SJ}

We have the following results, for their proof see the proof of \cite[Section~2.1]{Ha:Irreducible_polynomials}.

\begin{lemma}\label{lemma:inf_norm}
We have
$$
\int_{\mathbf{T}}|S_{\cJ}(\xi) | \d \xi =1.
$$
\end{lemma}

\begin{lemma}\label{lemma:generating_function}
We have
$$
S_{\cJ}(\xi)=
\left\{
\begin{array}{ll}
 q^{n- \#\cI} \e(\xi t^n) \prod_{i\in \cI}\e (a_it^i \xi) & \text{if }  \xi_{-i-1} =0 \text{ for all } i\not \in \cI,\\
 0 & \text{otherwise.}
\end{array}
\right.
$$
\end{lemma}

The following result is \cite[Lemma~2.7]{Ha:Irreducible_polynomials}. 

\begin{lemma}\label{lemma:char_fn_vanishes}
Let $a,g \in\F_q[t]$ be two polynomials with $\gcd(a,g)=1$ and write $g=g_0 t^k$, $\gcd(g_0,t)=1$. If $1<|g_0|\leq q^{\lceil (n-k) /(\#\cI+1)\rceil}-1$, then
$$
\cS_{\cJ}\left(\frac{a}{g}\right)=0.
$$
\end{lemma}
The original result \cite[Lemma~2.7]{Ha:Irreducible_polynomials} is stated correctly for $k=0$; however, for 
$k\geq 1$ it requires a minor revision.

\begin{proof}
   Write 
   $$
   \frac{a}{g}=\sum_{i\geq 0} \xi_{-i} t^{-i}.
   $$
Assume, that $S_\cJ(a/g)\neq 0$. Then, by Lemma~\ref{lemma:generating_function}, we have 
$$
\xi_{-i-1}=0 \quad \text{for } i\in \{0,\dots, n-1\}\setminus \cI
$$
In addition, as $g=g_0t^k$, we also have 
$$
\xi_0=\dots=\xi_{k-1}=0.
$$ 
Put 
$J=\lceil (n-k)/(\#\cI+1)\rceil$. By the pigeon-hole principle, there is at least $\lceil (n-\#\cI-k)/(\#\cI+1)\rceil\geq J-1$ consecutive indices in the range $\{k,\dots n\}$ where the coefficients $\xi_i$ vanish, say
$$
\xi_{r+1}=\dots=\xi_{r+J-1}=0.
$$
Then
$$
\left|\left\{t^r\cdot \frac{a}{g} \right\}\right|
\leq \left|\frac{1}{t^J}\right|\leq q^{-J},
$$
where $\{x\}$ denotes the fractional part of $x$. If $|g_0|
\neq 1$, the left hand side is at least $|g_0|^{-1}$ whence $|g_0|\geq q^{J}$.
\end{proof}

The following result is the analogue of Bourgain's result \cite[Lemma~3]{cite:Bourgain} proved by Ha \cite[Lemma~2.6]{Ha:Irreducible_polynomials}

\begin{lemma}\label{lemma:Bourgain}
 Let $k\leq n/2$. Then we have
 $$
 \sum_{\substack{|a|<|g|<q^k\\\gcd(aT,g)=1}} \left| 
 S_\cJ\left(\frac{a}{g}\right)\right|\leq \frac{q^n}{q^{\#\cI}} q^{2kC_1 \#\cI/n} 
 ,
 $$
 where $C_1=\kappa(n, 2k, \#\cI/n)\leq 2$ is given by
 $$
 \kappa(x,y,\delta) = 
 \left\{ 
 \begin{array}{cl}
  \frac{2}{\delta +1} & \text{if } 1<x/y<2\\
  \frac{2v}{(v+1)\delta +v-1} & \text{if } v-1<x/y<v \text{ for some integer $v$}\\
  1 & \text{if } y\mid x.
 \end{array}
 \right.
 $$
\end{lemma}

\section{Proof of Theorem~\ref{thm:2}}

First, we investigate the integral \eqref{eq:int} over the major arcs defined by \eqref{eq:major}.

As the polynomials $f$ in the summations $S(\xi;n,m)$ and $S_\cJ(\xi)$ are monic and of degree~$n$, we have for $\xi\in \fM$ having the form $\xi=a/g+\gamma$, $|\gamma|<q^{-n+\ell}$ that $\e(f\gamma)=\e(b\gamma)$ if $f\equiv b \mod \cR_{\ell}$, $b\in \cM(n)$ by the choice \eqref{eq:ell-def} of $\ell$.
Then 
\begin{align*}
 &\int_{\fM}S(\xi;n,m)\overline{S_\cJ(\xi)}\d \xi\\
 &=
 \sum_{\substack{a,g\\ |g|\leq q^{\kappa}}} \int_{|\gamma|<q^{-n+\ell}}S\left(\frac{a}{g}+\gamma;n,m\right)\overline{S_\cJ\left(\frac{a}{g}+\gamma\right)}\d \gamma
 \\
 &=\sum_{\substack{a,g\\ |g|\leq q^{\kappa}}}\sum_{\substack{b\bmod \cR_\ell\\ b\in\cM(n)}}
 \left(\sum_{\substack{f\in\cS(n,m)\\ f \equiv b \bmod \cR_\ell}} \e\left(\frac{af}{g}\right)
 \sum_{\substack{f\in\cJ\\ f \equiv b \bmod \cR_\ell}} 
 \e\left(-\frac{af}{g}\right)
 \right)
\int_{|\gamma|<q^{-n+\ell}} \e\left(b\gamma\right)\overline{ \e\left(b\gamma\right)}\d \gamma\\
&=\frac{1}{q^{n-\ell}}\sum_{\substack{a,g\\ |g|\leq q^{\kappa}}}\sum_{\substack{b\bmod \cR_\ell\\ b\in\cM(n)}}
 \left(\sum_{\substack{f\in\cS(n,m)\\ f \equiv b \bmod \cR_\ell}} \e\left(\frac{af}{g}\right)
 \sum_{\substack{f\in\cJ\\ f \equiv b \bmod \cR_\ell}} 
 \e\left(-\frac{af}{g}\right)
 \right).
\end{align*}

We split the sum above into two parts
\begin{equation}\label{eq:major_int}
   \int_{\fM}S(\xi;n,m)\overline{S_\cJ(\xi)}\d \xi = M_1+M_2 
\end{equation}
with 
$$
M_1=\sum_{0\leq k\leq \kappa} M_{1,k}
$$
where
$$
M_{1,k}= \frac{1}{q^{n-\ell}}  \sum_{\substack{|a|<t^k \\ t\nmid a}} \sum_{\substack{b\bmod \cR_\ell\\ b\in\cM(n)}}
 \left(\sum_{\substack{f\in\cS(n,m)\\ f \equiv b \bmod \cR_\ell}} \e\left(\frac{af}{t^k}\right)
 \sum_{\substack{f\in\cJ\\ f \equiv b \bmod \cR_\ell}} 
 \e\left(-\frac{af}{t^k}\right)
 \right)
$$
and
$$
M_2=\frac{1}{q^{n-\ell}}\sum_{\substack{a,g\\ g\neq t^k}}\sum_{\substack{b\bmod \cR_\ell\\ b\in\cM(n)}}
 \left(\sum_{\substack{f\in\cS(n,m)\\ f \equiv b \bmod \cR_\ell}} \e\left(\frac{af}{g}\right)
 \sum_{\substack{f\in\cJ\\ f \equiv b \bmod \cR_\ell}} 
 \e\left(-\frac{af}{g}\right)
 \right).
 $$

Concerning $M_{1,0}$,  
we have by Lemma~\ref{lemma:smooth_in_interval},
\begin{equation}\label{eq:M10}
\begin{split}
M_{1,0} =& \frac{1}{q^{n-\ell}}\sum_{\substack{b\bmod \cR_\ell\\ b\in\cM(n)}}
 \left(\sum_{\substack{f\in\cS(n,m)\\ f \equiv b \bmod \cR_\ell}} 1
 \sum_{\substack{f\in\cJ\\ f \equiv b \bmod \cR_\ell}} 
1
 \right) \\
 =&
 \frac{1}{q^{n-\ell}} \left(\frac{1}{q^\ell}\Psi(n,m)+ O\left(
 q^{\left(\frac{1}{2}+\varepsilon\right)n}
  \right)
 \right) \sum_{\substack{b\bmod \cR_\ell\\ b\in\cM(n)}}
 \sum_{\substack{f\in\cJ\\ f \equiv b \bmod \cR_\ell}} 
1\\
=&
\frac{\Psi(n,m)}{q^{\#\cI}} + 
O\left(
 q^{\left(\frac{1}{2}+\varepsilon\right)n - \#\cI+\ell}
  \right).
 \end{split}
 \end{equation}

Now consider the terms with $M_{1,k}$ with $k\geq 1$. Let $B=\cI\cap \{n-\ell,\dots, n-1\}$. Clearly, one can represent the residue classes of $\cR_\ell$ by the elements of $\cM(\ell)$. The contribution of those $b$ such that $b_{i-(n-\ell)}\neq \alpha_{i}$ for $i\in B$ is zero. Write
$$
\cB=\left\{b\in\cM(\ell): b_{j-(n-\ell)}= \alpha_{j}  \text{ for }  j\in B\right\} \mand \cI_\ell=\cI\cup \{n-\ell,\dots, n-1\}
$$
and for $b\in\cB$ write
$$
\cJ_b =\left\{f\in\cM(n): f\in \cJ \text{ and }  f\equiv b \mod \cR_\ell \right\}
$$
and
$$
\beta_i =
\left\{
\begin{array}{cl}
   \alpha_i  & \text{if $i\in \cI$} , \\
   b_{i-(n-\ell)}  & \text{if $i \in \{n-\ell,\dots, n-1\}\setminus \cI$ } .
\end{array}
\right.
$$

Then
\begin{equation}\label{eq:I_ell}
\#\cJ_b = q^{n-\#\cI_\ell}=q^{n-\#\cI-\ell +\#B}
\end{equation}
and 
$$
S_{\cJ_b}(\xi) =\sum_{\substack{f\in\cJ\\ f \equiv b \bmod \cR_\ell}}
 \e\left(f\xi\right).
$$
Whence
$$
M_{1,k}= \frac{1}{q^{n-\ell}}  \sum_{\substack{|a|<t^k \\ t\nmid a}} \sum_{\substack{b\bmod \cR_\ell\\ b\in\cM(n)}}
 \left(\sum_{\substack{f\in\cS(n,m)\\ f \equiv b \bmod \cR_\ell}} \e\left(\frac{af}{t^k}\right)
 \overline{S_{\cJ_b}\left(a/t^k\right)}
 \right)
$$

For $1\leq k\leq \kappa$, we recall that by Lemma \ref{lemma:major},  we have
\begin{equation}\label{eq:AP}
\sum_{\substack{f\in\cS(n,m)\\ f \equiv b \bmod \cR_\ell}} \e\left(\frac{af}{t^k}\right)
=
\frac{1}{q^{\ell-1}(q-1)}\Psi(n-k,m) - \frac{1}{q^{\ell}(q-1)}\Psi(n-k+1,m)
+E(n,m,k),
\end{equation}
with error term
$$
E(n,m,k) \ll q^{\left(\frac{1}{2}+\varepsilon\right)n+k/2}
.
$$

By Lemma~\ref{lemma:generating_function}, $S_{\cJ_b}(a/t^k)=0$ if $a_{i}\neq 0$ for some $k-i-1\not\in \cI_\ell$, $0\leq i <k$. As $\gcd(a,t)=1$, we have $a_0\neq 0$ and thus it is enough to consider $k$ with
 \begin{equation}\label{eq:k_M_k}
 k-1 \in \cI_\ell. 
   \end{equation}
 For such a $k$,  we have by Lemma~\ref{lemma:generating_function} and by \eqref{eq:AP} that
\begin{equation}\label{eq:M_1_fix_k}
\begin{split}
  M_{1,k}=  &\frac{1}{q^{n-\ell}}\sum_{b\in \cB}\sum_{\substack{f\in\cS(n,m)\\ f \equiv b \bmod \cR_\ell}} \e\left(\frac{af}{t^k}\right) \overline{S_{\cJ_b}(a/t^k)}\\
=& \frac{1}{q^{\#\cI_\ell-\ell}}\sum_{b\in \cB} \sum_{ \substack{a_{i}\in \Fq \\ k-i-1\in \cI_\ell\\
0\leq i < k\\ a_0\neq 0}}
\sum_{\substack{f\in S(n,m)\\ f \equiv b \bmod \cR_\ell}}
\e\left(\sum_{\substack{0\leq i<k\\ k-i-1\in \cI_\ell}} \frac{a_{i} t^{i}f}{t^k}\right)
\prod _{  j\in \cI_\ell}\e\left(-\beta_j t^j \left(\sum_{\substack{0\leq i<k \\ k-i-1\in \cI_\ell}} a_{i} t^{i-k}  \right) \right)\\
=&\frac{1}{q^{\#\cI_\ell-\ell}}
\left(
\frac{1}{q^{\ell-1}(q-1)}\Psi(n-k,m)-\frac{1}{q^\ell (q-1)}\Psi(n-k+1,m)
+E(n,m,k)
\right) \cdot \overline{S_\cI(n,k,\ell)}
\end{split}
\end{equation}
with
$$
S_\cI(n,k,\ell)=\sum_{b\in  \cB}
\sum_{ \substack{a_{i}\in \Fq \\ k-i-1\in \cI_\ell\\
0\leq i < k\\ a_0\neq 0}}\prod _{  j\in \cI_\ell}\e\left(\beta _j t^j \left(\sum_{\substack{0\leq i<k \\ k-i-1\in \cI_\ell}} a_{i} t^{i-k}  \right) \right).
$$

Considering $S_\cI(n,k,\ell)$, as $k-i-1\leq k\leq \kappa <n-\ell$, for the summation over $a_i$ it is enough to consider $i$ with $k-i-1\in \cI$. Moreover, for $j\geq n-\ell$ and $0\leq i<k$, we have $j +i-k\geq  n-\ell-k>0$, and thus
$$
\e\left(\beta_j t^j \left(\sum_{\substack{0\leq i<k \\ k-i-1\in \cI_\ell}} a_{i} t^{i-k}  \right)\right)=1,
$$
i.e., the terms in $S_\cI(n,k,\ell)$ do not depend on the choice of $b\in\cB$. Therefore
\begin{align}\label{eq:S_I}
S_\cI(n,k,\ell) 
&=
q^{\ell-\#B}
\sum_{ \substack{a_{i}\in \Fq \\ k-i-1\in \cI\\
0\leq i < k\\ a_0\neq 0}}\prod _{  j\in \cI}\e\left(\alpha _j t^j \left(\sum_{\substack{0\leq i<k \\ k-i-1\in \cI}} a_{i} t^{i-k}  \right) \right)
\\  \notag
&=q^{\ell-\#B}\prod_{ \substack{j\in \cI\\
0\leq j < k-1}}
\left(\sum_{a_{k-j-1}\in \Fq}
\e\left( \alpha_{j}a_{k-j-1} t^{-1}  \right) \right)\left(\sum_{a_{0}\in \Fq^*}
\e\left( \alpha_{k-1}a_{0} t^{-1}  \right) \right)
\\ \notag
&=\sigma_k q^{\ell-\#B+\#\cI\cap\{0,\dots, k-2\}}\mathbf{1}(\alpha_j = 0, j\in \cI\cap\{0,\dots, k-2\} ),
\end{align}
where
$$
\sigma_k = 
\left\{
\begin{array}{cl}
  -1 &  \alpha_{k-1}\neq 0 , \\
  q-1   & \text{otherwise} .
\end{array}
\right.
$$

Write $I= \{i_0<i_1<\dots< i_\nu<\dots \}$ with $a_{i_0}=a_{i_1}=\dots =a_{i_{\nu-1}}=0\neq a_{i_\nu}$.
Then taking summation over $0\leq k\leq \kappa$ we get by \eqref{eq:M10}, \eqref{eq:I_ell},  \eqref{eq:k_M_k}, \eqref{eq:M_1_fix_k} and \eqref{eq:S_I}, that
\begin{equation}\label{eq:M1}
\begin{split}
M_1 =&    \frac{\Psi(n,m)}{q^{\#\cI}} + 
O\left(
 q^{\left(\frac{1}{2}+\varepsilon\right)n - \#\cI+\ell}
  \right)
+\sum_{\substack{1\leq k\leq \kappa\\ k-1\in\cI_\ell}}M_{1,k}
 \\
=&    \frac{\Psi(n,m)}{q^{\#\cI}} + 
O\left(
 q^{\left(\frac{1}{2}+\varepsilon\right)n - \#\cI+\ell}
  \right)
 \\
&+ (q-1) \sum_{\substack{0\leq j<\nu \\ i_j<\kappa}} 
\frac{1}{q^{\#\cI-j-\ell}}  
\Bigg(
\frac{1}{q^{\ell-1}(q-1)}\Psi(n-i_j-1,m)-\frac{1}{q^\ell(q-1)}\Psi(n-i_j,m) \\
&\qquad\qquad\qquad\qquad\qquad\qquad
+E(n,m,i_j+1)
\Bigg)
\\
&-\mathbf{1}(i_\nu<\kappa)
\frac{1}{q^{\#\cI-\nu-\ell}}  \Bigg(
\frac{1}{q^{\ell-1}(q-1)}\Psi(n-i_\nu-1,m)-\frac{1}{q^{\ell}(q-1)}\Psi(n-i_\nu,m) 
\\
&\qquad\qquad\qquad\qquad\qquad\qquad
+E(n,m,i_\nu+1)
\Bigg)\\
=&     \frac{1}{q^{\#\cI}}\Psi(n,m) +
\frac{q^n}{q^{\#\cI}}
\left(
\Lambda(n,m,\cI)
+O\left(
q^{\left(-\frac{1}{2}+\varepsilon\right)n +\kappa/2+\ell + \nu}
\right)
\right) .
\end{split}
\end{equation}

Next, we consider $M_2$. For the summation $g$, write $g=g_0t^k$ with $t\nmid g_0$.  By Lemma~\ref{lemma:char_fn_vanishes}, if $g_0<q^{\lceil (n-k)/(\#\cI +1)\rceil}$, $S_{\cJ}(a/g)$ vanishes. Put
\begin{equation}\label{eq:R1_def}
R_k=\left\lceil \frac{n-k}{\#\cI +1}\right\rceil
\end{equation}
Then, we can write
\begin{equation}\label{eq:M2_split}
M_2= \frac{1}{q^{n-\ell}}  \sum_{ k=0}^\kappa \sum_{r=R_k}^{\kappa-k} M_2(r,k)
\end{equation}
with
$$
M_2(r,k)= \sum_{\substack{\substack{g_0\in \cM(r)\\t\nmid g_0}}}\sum_{\substack{|a|<|g_0t^k|\\ \gcd(a,g_0t)=1}}\sum_{\substack{b\bmod \cR_\ell\\ b\in\cM(n)}}
 \left(\sum_{\substack{f\in\cS(n,m)\\ f \equiv b \bmod \cR_\ell}} \e\left(\frac{af}{g_0t^k}\right)
 \sum_{\substack{f\in\cJ\\ f \equiv b \bmod \cR_\ell}} 
 \e\left(-\frac{af}{g_0t^k}\right)
 \right).
$$

By Lemma~\ref{lemma:major2},
\begin{equation*}
\begin{split}
M_2(r,k)
&= \sum_{\substack{\substack{g_0\in \cM(r)\\t\nmid g_0}}}\sum_{\substack{|a|<|g_0t^k|\\ \gcd(a,g_0t)=1}}\sum_{\substack{b\bmod \cR_\ell\\ b\in\cM(n)}} \left(\sum_{d \mid g_0t^k}   \frac{\mu(g_0t^k/d)}{q^\ell\Phi(g_0t^k/d)}  \Psi_{g_0t^k/d}(n-\deg d,m)
+R(r,k) \right)\\
& \qquad \cdot
\sum_{\substack{f\in\cJ\\ f \equiv b \bmod \cR_\ell}} 
 \e\left(-\frac{af}{g_0t^k}\right)\\
 &\ll \sum_{\substack{\substack{g_0\in \cM(r)\\t\nmid g_0}}} \left(\sum_{d \mid g_0t^k}   \frac{|\mu(g_0t^k/d)|}{q^\ell\Phi(g_0t^k/d)}  \Psi_{g_0t^k/d}(n-\deg d,m)
+R(r,k) \right)\\
& \qquad \cdot
\left|\sum_{\substack{|a|<|g_0t^k|\\ \gcd(a,g_0t)=1}}S_J\left(\frac{a}{g_0t^k}\right)\right|
\end{split}    
\end{equation*}
with error term
\begin{equation}\label{eq:M2_error}
R(r,k)\ll  
q^{\left(\frac{1}{2}+\varepsilon\right)n +\frac{r+k}{2}} 
\end{equation}

Since $u\geq 100$ by \eqref{eq:prop:condition1}, we have
by \eqref{eq:Psi2}, \eqref{eq:phi} and Lemma~\ref{lemma:rho} that
\begin{align*}
&\sum_{d \mid g_0t^k}   \frac{|\mu(g_0t^k/d)|}{q^\ell\Phi(g_0t^k/d)}  \Psi_{g_0t^k/d}(n-\deg d,m)\\
\leq &\sum_{d \mid g_0t^k}   \frac{1}{q^\ell\Phi(g_0t^k/d)}  \Psi(n-\deg d,m)\\
 \ll &\sum_{d \mid g_0t^k}   \frac{q^{n-\deg d}}{q^{\ell}\Phi(g_0t^k/d)} \rho\left(\frac{n-\deg d}{m}-\frac{n-\deg d}{2m^2}\right)\exp\left(O_\varepsilon \left(\frac{u\log^2(u+1)}{m^2} \right)\right) \\
 \ll& q^{n-\ell}\rho\left(\frac{n}{m}-\frac{n
 }{2m^2}\right)\exp\left(O_\varepsilon \left(\frac{u\log^2(u+1)}{m^2} \right)\right)  \sum_{d \mid g_0t^k}    \frac{(u \log u )^{\deg d/m}}{|d|\Phi(g_0t^k/d)}
\\
 \ll & q^{n-\ell-r-k}\rho\left(\frac{n}{m}-\frac{n}{2m^2}\right)\exp\left(O_\varepsilon \left(\frac{u\log^2(u+1)}{m^2}\right)\right) \\
 &\qquad  \cdot
\left(u\log u\right)^{(r+k)/m}
\left(1+\log_q (r+k) \right)\tau(g_0t^k)
.
\end{align*}
Lemma~ \ref{lemma:number_of_divisors} yields $\tau(g_0t^k)\ll q^{(2+\varepsilon)(r+k)/\log(r+k)}\ll q^{(r+k)/5}$. Moreover,
$$
\left(u\log u\right)^{(r+k)/m}\leq
\exp\left\{\frac{11\log u}{10} \cdot \frac{r+k}{m}\right\}\leq \exp\left\{\frac{11(r+k)}{10}\cdot\frac{\log n}{m}\right\}\leq q^{3\frac{r+k}{5}}
$$
by \eqref{eq:condition1} and 
 \eqref{eq:thm1:range}. Thus
\begin{align*}
&\sum_{d \mid g_0t^k}   \frac{\mu(g_0t^k/d)}{q^\ell\Phi(g_0t^k/d)}  \Psi_{g_0t^k/d}(n-\deg d,m)\\
& \ll q^{n-\ell-(r+k)/5}\rho\left(\frac{n}{m}-\frac{n}{2m^2}\right)e^{T(n,m)} 
\left(1+\log_q (r+k) \right)
\end{align*}
with error term 
$$
T(n,m)\ll_\varepsilon  \frac{u\log^2(u+1)}{m^2} 
$$
Whence
\begin{equation}\label{eq:M2-partial}
\begin{split}
M_2(r,k)
&= 
\left(
q^{n-\ell-(r+k)/5}\rho\left(\frac{n}{m}-\frac{n}{2m^2}\right)e^{T(n,m)} (r+k) \left(1+\log_q (r+k) \right)
+R(r,k)
\right)\\
&\qquad \cdot \sum_{\substack{\substack{g_0\in \cM(r)\\t\nmid g_0}}}\left|\sum_{\substack{|a|<|g_0t^k|\\ \gcd(a,g_0t)=1}} 
S_\cJ\left(\frac{a}{g_0t^k}\right)\right|.
\end{split}
\end{equation}

In order to handle the second term, write 
$$
\frac{af}{g_0t^k}=\frac{a_1f}{g_0}+\frac{a_2f}{t^k} \quad \text{with} \quad \gcd(a_1,g_0)=\gcd(a_2,t)=1.
$$
Then we have by Lemmas~\ref{lemma:Ram} and \ref{lemma:Bourgain} that
\begin{align}\label{eq:k=01} \notag
   & 
\sum_{\substack{g_0\in\cM(r)\\ t\nmid g_0}}
\left| \sum_{\substack{|a|<|g_0t^k|\\ \gcd(a,g_0t)=1}}
S_\cJ\left(\frac{a}{g_0t^k}\right)\right|\\ \notag
=& \sum_{\substack{ g_0\in\cM(r)\\ t \nmid g_0}}\left| \sum_{\substack{|a_1|<|g_0|\\ \gcd(a_1,g_0)=1 }} \sum_{\substack{|a_2|<|t^k|\\ \gcd(a_2,t)=1}}
   \sum_{\substack{f\in\cJ}} \e\left(-\frac{a_1f}{g_0}-\frac{a_2f}{t^k}\right)\right|\\ \notag
  =& \sum_{\substack{g_0\in\cM(r)\\ t \nmid g_0}}\left| \sum_{\substack{|c|<|t^k|\\ \gcd(c,t)=1}} \sum_{\substack{|a_2|<|t^k|\\ \gcd(a_2,t)=1}} \e\left(-\frac{a_2c}{t^k}\right)\sum_{\substack{|a_1|<|g_0|\\ \gcd(a_1,g_0)=1 }}  \sum_{\substack{f\in\cJ\\ f \equiv c \bmod t^k}} \e\left(-\frac{a_1f}{g_0}\right)\right|\\ \notag
=& 
  \sum_{\substack{g_0\in\cM(r)\\ t \nmid g_0}}
  \left|\sum_{\substack{|c|<|t^k|\\ \gcd(c,t)=1}}
  \mu(t^k)
  \sum_{\substack{|a_1|<|g_0|\\ \gcd(a_1,g_0)=1 }}  \sum_{\substack{f\in\cJ\\ f \equiv c \bmod t^k}} \e\left(-\frac{a_1f}{g_0}\right)
  \right|\\
=& |\mu(t^k)|\sum_{\substack{g_0\in\cM(r)\\ t \nmid g_0}}
  \left|\sum_{\substack{|c|<|t|\\ \gcd(c,t)=1}}
  \mu(t^k)
  \sum_{\substack{|a_1|<|g_0|\\ \gcd(a_1,g_0)=1 }}  \sum_{\substack{f\in\cJ\\ f \equiv c \bmod t^k}} \e\left(-\frac{a_1f}{g_0}\right)
  \right|.
\end{align}
Clearly, \eqref{eq:k=01} is zero, unless, 
\begin{equation}\label{eq:samll_k}
k\leq 1.
\end{equation}
If $k=0$, we have by Lemma~\ref{lemma:Bourgain} that \eqref{eq:k=01} is
\begin{equation}\label{eq:k0}
 O\left(  \frac{q^{n}}{q^{\#\cI}}q^{2rC_1\#\cI/n}\right).
\end{equation}
If $k=1$, for all $|c|\leq |t|$, we can add the constrains to $\cJ$ that the constant term of $f$ is $c$ to derive that \eqref{eq:k=01} is
\begin{equation}\label{eq:k1}
 O\left(  \frac{q^{n}}{q^{\#\cI}}q^{2rC_1(\#\cI+1)/n}
 \right).
\end{equation}
 Thus \eqref{eq:k0} and \eqref{eq:k1} yield that by \eqref{eq:M2-partial} we have
\begin{equation}\label{eq:M2-partial2}
\begin{split}
    M_2(r,k)
&\ll
\Bigg(
q^{n-\ell-(r+k)/5}\rho\left(\frac{n}{m}-\frac{n}{2m^2}\right)
e^{T(n,m)}
(r+k) \left(1+\log_q (r+k)\right)
+R(r,k)
\Bigg)\\
& \quad \cdot |\mu(t^k)| \frac{q^{n}}{q^{\#\cI}}q^{2rC_1(\#\cI+k)/n}.
\end{split}
\end{equation}

Write
$$
C_2=\frac{1}{5}-2C_1 \frac{\#\cI+1}{n}.
$$ 
As $\#\cI< n/24 $, we have $2C_1(\#\cI+1)/n\leq 1/6$ and thus $C_2\geq 1/30$.  Then we have  by \eqref{eq:M2_split}, \eqref{eq:M2_error},\eqref{eq:samll_k}  and \eqref{eq:M2-partial2} that
\begin{equation}\label{eq:M2final}
\begin{split}
    M_2 &
\ll \frac{q^{n}}{q^{\#\cI}} \left(\rho\left(\frac{n}{m}-\frac{n}{2m^2}\right)
e^{T(n,m)}
\sum_{ r\geq R_1 } \frac{r \left(1+\log_q r \right)}{q^{C_2r}} + O\left(
\kappa q^{\left(-\frac{1}{2}+\varepsilon\right)n+\kappa +\frac{1}{2}+\ell}
\right) \right)\\
&\ll \frac{q^{n}}{q^{\#\cI}} \left(\rho\left(\frac{n}{m}-\frac{n}{2m^2}\right)
e^{T(n,m)}
\frac{R_1 \left(1+\log_q R_1 \right)}{q^{C_2R_1}}  + O\left(
\kappa q^{\left(-\frac{1}{2}+\varepsilon\right)n+\kappa +\frac{1}{2}+\ell}
\right)   \right)
,
\end{split}
\end{equation}
where $R_1$ is defined by \eqref{eq:R1_def}.

Then, combining \eqref{eq:major_int}, \eqref{eq:M1} and \eqref{eq:M2final}, we get that the contribution of the major arcs is

\begin{equation}\label{eq:major-final}
\begin{split}
\int_{\fM}S(\xi;n,m)\overline{S_\cJ(\xi)}\d \xi &= 
\frac{\Psi(n,m)}{q^{\# \cI}} 
\left(1 + O\left(
e^{T(n,m)}
\frac{\delta^{-1} \log \delta}{q^{C /\delta}}
\right)
\right) \\
&\quad
+\frac{q^{n}}{q^{\#\cI}}
\left(
\Lambda(n,m,\cI) 
+ 
O\left( 
\kappa q^{\left(-\frac{1}{2}+\varepsilon\right)n+\kappa+\ell+\nu+\frac{1}{2}}
\right)
\right)
\end{split}
\end{equation}
for some positive $C$.

As $\nu\leq i_\nu<\kappa$, we get by choosing
\begin{equation}\label{eq:ell}
    \ell=\kappa+1,
\end{equation}
that the error term here is
\begin{equation}\label{eq:major-final-error}
   O\left( 
\kappa q^{\left(-\frac{1}{2}+\varepsilon\right)n+
\kappa+\ell +\nu+\frac{1}{2}}e^{\varepsilon(\ell+\kappa)}
\right)=
   O\left( 
\kappa q^{\left(-\frac{1}{2}+\varepsilon\right)n+3\kappa+\frac{3}{2}}
\right).
\end{equation}

Next, we consider the integral along the minor arcs $\fm$
\begin{equation*}  \int_{\fm}S(\xi;n,m)\overline{S_\cJ(\xi)}\d \xi.
\end{equation*}
By Lemma~\ref{lemma:dio}, write    $\xi\in\mathbf{T}$ as
$$
\xi=\frac{a}{g}+\gamma, \quad \gcd(a,g)=1, |g|<q^{n/2}, |\gamma|\leq \frac{1}{|g|q^{n/2}}.
$$
If $\xi\in \fm$, then either $q^{\kappa}\leq|g|<q^{n/2}$ or $|g|<q^{\kappa}$ and $q^{-n+\ell}\leq |\gamma|\leq 1/(|g|q^{n/2})$. In the former case, we estimate $S(\xi;n,m)$ by the first bound of Lemma~\ref{lemma:minor_arcs}
\begin{align*}
S(\xi,m,n)&\ll 
     m^2 n^{1/2}\left(\frac{q^{n}}{|g|^{1/2}}+q^{3n/4+m/4}+|g|^{1/2}q^{(n+1)/2} \right)\\
     &\ll
    m^2 n^{1/2}\left(q^{n-\kappa/2}+q^{3n/4+m/4}  \right).
\end{align*}
In the second case, we use the  second bound of Lemma~\ref{lemma:minor_arcs}
\begin{align*}
S(\xi,m,n)&\ll  m^2 n^{1/2}\left(
|g|^{1/2}q^{3n/4+m/4+1/2}+\|g\gamma\|^{1/2}q^n  +\frac{q^{(n+1)/2}}{\|\gamma\|^{1/2}}     \right)\\
&\ll
m^2 n^{1/2}\left(
q^{3n/4+m/4+\kappa/2+1/2}  +q^{n-(\ell-1)/2}     \right).
\end{align*}
We get in both cases by \eqref{eq:kappa} and \eqref{eq:ell} that
$$
S(\xi,m,n)\ll
m^2 n^{1/2}  q^{n-\kappa/2} .
$$
Thus by Lemma~\ref{lemma:inf_norm}, we have
\begin{equation}\label{eq:minor-final}
\begin{split}
\int_{\fm}S(\xi;n,m)\overline{S_\cJ(\xi)}\d \xi&
 \ll  m^2 n^{1/2}  q^{n-\kappa/2} 
 \int_{\fm}|S_\cJ(\xi)|\d \xi\\
 &\ll  m^2 n^{1/2}q^{n-\kappa/2}\int_{\mathbf{T}}|S_\cJ(\xi)|\d \xi\\
&\ll m^2 n^{1/2}q^{n-\kappa/2}.
\end{split}
\end{equation}

By replacing $\varepsilon$ if needed in \eqref{eq:major-final}, 
Theorem~\ref{thm:2} follows from \eqref{eq:major-final}, \eqref{eq:major-final-error} 
and \eqref{eq:minor-final}.

\section{Proof of Theorem~\ref{thm:1}}\label{sec:proof}

Theorem~\ref{thm:1} follows immediately from Theorem~\ref{thm:2} by noting that $\Lambda(n,m,\cI)=0$ if $0\in \cI$ and $\alpha_0\neq 0$.

\section{Proof of Corollaries \ref{cor:1} and \ref{cor:2}}

Corollary \ref{cor:1} follows directly from Theorem~\ref{thm:1}.

\smallskip

For Corollary \ref{cor:2},
consider $\Lambda(n,m,\cI,\kappa)$. For its terms we have 
\begin{align*}
    \frac{1}{q^{n-j-1 }}  
\left(
\Psi(n-i_j-1,m)-\frac{1}{q}\Psi(n-i_j,m)
\right)\ll \frac{1}{q^{i_j-j}}
\left(
\rho\left(\frac{n-i_j-1}{m}-\frac{n-i_j-1}{2m^2}\right)
\right).
\end{align*}
As $0\not\in \cI$, we have $i_j\geq j+1$, thus
\begin{align*}
\Lambda(n,m,\cI,\kappa) \ll \frac{1}{q} 
\end{align*}
provided, that $m$ and $n$ are fixed.

\section*{Acknowledgment}

The author wishes to thank Ofir Gorodetsky and Eugenijus Manstavi\u{c}ius for useful discussions and valuable comments. He also thanks the anonymous referees for the valuable comments and calling attention to \cite{cite:Bhow_et_al}, which allowed us to improve Lemma~\ref{lemma:char_sum}.

\section*{Funding}
The author was partially supported by NRDI (National Research Development and Innovation Office, Hungary) grant FK 142960 and by the János Bolyai Research Scholarship of the Hungarian Academy
of Sciences.

\section*{Data Availability Statement} 
There is no data associated with this work.

\section*{Conflict of Interest} 
The authors have no conflicts of interest to declare

\end{document}